\documentclass{amsart}

\usepackage{amssymb, amsmath,  amsfonts} %, mathabx}
\usepackage{color}
\usepackage[all,cmtip]{xy}

\newtheorem{thm}{Theorem}[section]
\newtheorem{prop}[thm]{Proposition}
\newtheorem{lemma}[thm]{Lemma}
\newtheorem{cor}[thm]{Corollary}
\newtheorem{defn}[thm]{Definition}%[section]
\newtheorem{lem-defn}[thm]{Lemma/Definition}%[section]
\newtheorem{example}[thm]{Example}
\newtheorem{remark}[thm]{Remark}

\numberwithin{equation}{section}

 \numberwithin{equation}{section}

\newcommand{\C}{{\mathbb C}}

\newcommand{\sP}{{\mathscr{P}}}
\newcommand{\sQ}{{\mathscr{Q}}}

\newcommand{\comment}[1]{}

\renewcommand{\emptyset}{\varnothing}
\renewcommand{\tilde}{\widetilde}
\DeclareFontFamily{OT1}{pzc}{}
\DeclareFontShape{OT1}{pzc}{m}{it}%
              {<-> s * [0.900] pzcmi7t}{}
\DeclareMathAlphabet{\mathpzc}{OT1}{pzc}%
                                 {m}{it}
\DeclareMathAlphabet      {\mathsf}{OT1}{cmss}{m}{n}
\usepackage[mathscr]{eucal}

\usepackage{graphicx}
\makeatletter
\DeclareRobustCommand\widecheck[1]{{\mathpalette\@widecheck{#1}}}
\def\@widecheck#1#2{%
   \setbox\z@\hbox{\m@th$#1#2$}%
   \setbox\tw@\hbox{\m@th$#1%
      \widehat{%
         \vrule\@width\z@\@height\ht\z@
         \vrule\@height\z@\@width\wd\z@}$}%
   \dp\tw@-\ht\z@
   \@tempdima\ht\z@ \advance\@tempdima2\ht\tw@ \divide\@tempdima\thr@@
   \setbox\tw@\hbox{%
      \raise\@tempdima\hbox{\scalebox{1}[-1]{\lower\@tempdima\box\tw@}}}%
   {\ooalign{\box\tw@ \cr \box\z@}}}
\makeatother

\begin{document}{\allowdisplaybreaks[4]}

\title{Effective good divisibility of rational homogeneous varieties}

%    Information for first author

\author{Haoqiang Hu}
\address{School of Mathematics, Sun Yat-sen University, Guangzhou 510275, P.R. China}
\email{huhq6@mail2.sysu.edu.cn}

\author{Changzheng Li}
 \address{School of Mathematics, Sun Yat-sen University, Guangzhou 510275, P.R. China}
\email{lichangzh@mail.sysu.edu.cn}

\author{Zhaoyang Liu}
\address{School of Mathematics, Sun Yat-sen University, Guangzhou 510275, P.R. China}
\email{liuchy73@mail2.sysu.edu.cn}
\thanks{ %The first author  is supported in part by %a RGC research grant from the Hong Kong Government.
%The second author  is  supported in part %by KRF-2007-341-C00006.
 }
 \thanks{While this manuscript was almost finished,  Mu\~{n}oz, Occhetta and Sol\'{a} Conde posted a preprint \cite{MOS22} independently proving  Theorems 1.1 and 1.2 for  rational homogeneous varieties of classical type.   Our proof is completely different from theirs, and our  main results hold for all Lie types.}

\date{%January 1, 2001 and, in revised form, June 22, 2001.
      }

%\dedicatory{This paper is dedicated to our advisors.}

%\keywords{Lie superalgebra, spin  representations}

%\subtitle{}

\begin{abstract}
We compute the effective good divisibility of a rational homogeneous variety, extending an earlier result for complex Grassmannians by Naldi and Occhetta.
Non-existence of nonconstant morphisms to rational homogeneous varieties of classical Lie type are obtained as applications.
\end{abstract}

\maketitle
%\tableofcontents

\section{Introduction}\label{S:intro}

The notion of \textit{effective  good divisibility} $\mbox{e.d.}(X)$ of a complex projective manifold $X$ was introduced by   Mu\~{n}oz, Occhetta and Sol\'{a} Conde \cite{MOS20}, refining the notion of good divisibility introduced by Pan \cite{Pan}. It is the upper bound of the total degree $i+j$  of two nonzero effective classes $x_i\in H^{2i}(X)$ and $x_j\in H^{2j}(X)$ with $x_i\cup x_j\neq 0$.  It has nice applications on the nonexistence of nonconstant morphisms.
As shown by Naldi and Occhetta \cite{NaOc}, the  {effective  good divisibility} of a complex Grassmannian $Gr(m, n+1)$ is equal to $n+1$. It is then quite natural to ask for this quantity for a general rational homogeneous variety $G/P$, the quotient of a complex simple and simply-connected Lie group $G$ by its parabolic subgroup $P$.

The question on $\mbox{e.d.}(G/P)$ is actually closely related with a     problem of characterizing the  degree of quantum variables in   quantum product, in earlier and extensive  studies   of quantum cohomology.
The quantum cohomology ring $QH^*(G/P)=(H^*(G/P)\otimes \mathbb{Z}[\mathbf{q}], \star)$ is a deformation of the classical cohomology $H^*(G/P)$, by incorporating genus zero, three-pointed Gromov-Witten invariants.  It has a canonical basis of Schubert classes $[X_u]$, and the evaluation $[X_u]\star [X_v]|_{\mathbf{q}=\mathbf{0}}$ coincides with $[X_u] \cup [X_v]$.  As shown in \cite{FuWo}, the quantum product    $[X_u]\star [X_v]$ never vanishes, leading to the characterization of
  the degree of $\mathbf{q}$  and particularly the minimal degree $\mathbf{d}_{\rm min}$ in such individual product   asked two decades ago.  The  characterization was   provided for complex Grassmannians therein,  for (co)minuscule Grassmannians in \cite{CMP, BuMi} and for Grassmannians of   types  B  and  C  in \cite{ShWi}. There has also been extended study in the context of quantum $K$-theory \cite{BCLM, BCMP}.  With this terminology,  $\mbox{e.d.}(G/P)$ is equivalent to the upper bound of the total degree of  all Schubert classes $[X_u]$ and  $[X_v]$ that satisfy $\mathbf{d}_{\rm min}=0$ in $[X_u]\star [X_v]$. As we will see, the viewpoint from quantum cohomology is particularly useful in the study of $\mbox{e.d.}(G/P)$ in the case of minuscule Grassmannians of classical Lie type.

Let $\mathcal{D}$ denote the Dynkin diagram  of $G$ with the set $\Delta$  of nodes, for which we take
 the  same numbering  as on page 58 of \cite{Hump}.
  Parabolic subgroups $P$ of $G$   that contain  a fixed Borel subgroup $B$  are in one-to-one correspondence with the subsets $\Delta_P$ of $\Delta$. Let $P_m$ be the maximal parabolic subgroup that corresponds  to the subset $\Delta\setminus\{\alpha_m\}$. By $\mathcal{D}(m)$ we mean the Grassmannian $G/P_m$ with $G$ being of type $\mathcal{D}$. Then  $\mbox{A}_n(m)$ is the complex Grassmannian $Gr(m, n+1)$, and $\mathcal{D}_m$ can  be realized as isotropic Grassmannians for other classical   types.
As a first main result, we obtain the following, where the known cases $\mbox{A}_n(m)$ and $G_2(m)$ are    included for completeness.
\begin{thm}\label{mainthm1}
   The  {effective  good divisibility} {\upshape $\mbox{e.d.}(\mathcal{D}(m))$}
   is  given in   Table \ref{tabed}.
   %%equals {\upshape ${\rm h}(\mathcal{D})$} if $\mathcal{D}$ is of type  {\upshape $\mbox{D}_{n+1}$}  and    $k\in \{1, n, n+1\}$, or   {\upshape ${\rm h}(\mathcal{D})-1$} otherwise.

  \begin{table}[h]
  \caption{Effective  good divisibility $\mbox{e.d.}$ of $\mathcal{D}(m)$}\label{tabed}
 { \upshape

\begin{tabular}{|c||p{0.9cm}|p{0.9cm}|p{0.9cm}|c|c|p{0.86cm}|p{0.86cm}|p{0.9cm}|}
     \hline \rule{0pt}{10pt}
     % after \\: \hline or \cline{col1-col2} \cline{col3-col4} ...
     $\mathcal{D}$  &  &  &   & \multicolumn{2}{|c|}{ \raisebox{-0.0ex}[0pt]{$\mbox{D}_{n+1}(m)$}}  &\multicolumn{2}{|c|}{  {$\mbox{F}_{4}$}} &     \\  \cline{1-1}\cline{5-8}
     {$m$}   &  \centering{\raisebox{2ex}[0pt]{$\mbox{A}_n(m)$}}  & \centering{\raisebox{2ex}[0pt]{$\mbox{B}_n(m)$}}   & \centering{\raisebox{2ex}[0pt]{$\mbox{C}_n(m)$}}   &     $1, n, n+1 $ &$2\leq m<n$ &\hspace{0.05cm} $1,4$ & \hspace{0.050cm}  $2,3$&  {\centering{\raisebox{2ex}[0pt]{$\mbox{G}_2(m)$}}} \\ \hline \rule{0pt}{14pt}
     %$\mbox{e.d.}(\mathcal{D}(m))$
     e.d. & \centering{$n$}  & \centering{$2n-1$}  & \centering{$2n-1$}  & $2n-1$  & $2n$ & \hspace{0.1cm}  $12$ & \hspace{0.1cm}  $14$ & \hspace{0.2cm}  $5$   \\ \hline %\rule{0pt}{14pt}
     % ${\rm h}(\mathcal{D})$ &\centering{ $n+1$}  & \centering{$2n$}  & \centering{$2n$}  & \multicolumn{2}{|c|}{$2n$}   \\  \hline
       \end{tabular}
    \bigskip
  \iffalse \begin{tabular}{|p{1.4cm}||p{1.4cm}|p{1.4cm}|p{1.4cm}|c|c|}
     \hline \rule{0pt}{12pt}
     % after \\: \hline or \cline{col1-col2} \cline{col3-col4} ...
      &  &  &   & \multicolumn{2}{|c|}{ \raisebox{-0.0ex}[0pt]{$\mbox{D}_{n+1}(k)$}}       \\  \cline{5-6} \rule{0pt}{12pt}
     \raisebox{2ex}[0pt]{$\mathcal{D}(k)$}   &  \centering{\raisebox{2ex}[0pt]{$\mbox{A}_n(k)$}}  & \centering{\raisebox{2ex}[0pt]{$\mbox{B}_n(k)$}}   & \centering{\raisebox{2ex}[0pt]{$\mbox{C}_n(k)$}}   &     $k\in \{1, n, n+1\}$ &$2\leq k<n$  \\ \hline \rule{0pt}{12pt}
     $\mbox{e.d.}(\mathcal{D}(m))$  & \centering{$n$}  & \centering{$2n-1$}  & \centering{$2n-1$}  & $2n-1$  & $2n$  \\ \hline %\rule{0pt}{14pt}
     % ${\rm h}(\mathcal{D})$ &\centering{ $n+1$}  & \centering{$2n$}  & \centering{$2n$}  & \multicolumn{2}{|c|}{$2n$}   \\  \hline
       \end{tabular}
    \bigskip
  \fi
       \begin{tabular}{|c||c|c|c|c|c|c|c|c|c|c|c|c|c|c|c|c|}
   \hline \rule{0pt}{12pt}
   % after \\: \hline or \cline{col1-col2} \cline{col3-col4} ...
    $\mathcal{D}$ & \multicolumn{3}{|c|}{$\mbox{E}_6$}  &\multicolumn{5}{|c|}{$\mbox{E}_7$}   & \multicolumn{6}{|c|}{$\mbox{E}_8$} \\ \hline \rule{0pt}{12pt}
        $m$ & $1,6$ & $2,3,5$ &4  & 1 & 2,6 & 3 & 4,5  & 7 & 1& $2,3,5$ & 4 & 6 &7  & 8 \\ \hline \rule{0pt}{12pt}
        % $\mbox{e.d.}(\mathcal{D}(m))$
      e.d.    &12 & 14 &15  & 22 & 23 & 24 & 25  & 19 & 46 & 50  & {51} &48 & 45  &40 \\
   \hline
 \end{tabular}
  }   \end{table}
\end{thm}
\vspace{-0.6cm}

\noindent The above theorem will be proved for classical types in section 3.4 and for exceptional   types in section 4.   As from Table \ref{tabed}, the effective good divisibility of Grassmannians $G/P_m$   may equal each other for distinct $m$. For instance,  $\mbox{e.d.}(\mbox{E}_7(2))=\mbox{e.d.}(\mbox{E}_7(6))=23$;  $\mbox{e.d.}(\mathcal{D}(m))$ is   independent of $m$, whenever  $\mathcal{D}$ is of type A, B, C or $\mbox{G}_2$.

The nonvanishing of $[X_u]\cup[X_v]$ in $H^*(G/P)$ is equivalent to the nonempty intersection of Schubert variety $X_u$ and another Schubert variety $X^{v^\sharp}$   opposite to $X_v$.  This is  further equivalent to the property $u\leq v^{\sharp}$ with respect to the Bruhat order. Then we can reduce it to the Bruhat order among Grassmannian permutations. Such reduction  is well known  for type Lie A (see for instance \cite[ exercise 8 on page 175]{Fult}). The reduction for general Lie type should have also been known  to the experts, while we treat it as a special case of \cite[Theorem 5]{BCLM}, for not being aware of a precise reference elsewhere. As a consequence,
we obtain the following.
\begin{thm}\label{mainthm2}
   The  {effective  good divisibility} of $G/P$ of any  type $\mathcal{D}$ is given by
   {\upshape $$\mbox{e.d.}(G/P)=\min\{  \mbox{e.d.}(\mathcal{D}(m))\mid m\in \Delta\setminus \Delta_P\}.  $$}
  \end{thm}
 \noindent  In particular  for any  complete flag manifold $G/B$,  the next table of {\upshape $\mbox{e.d.}(G/B)$} follows immediately from the above theorem and Table \ref{tabed}.  Therein we provide the Coxeter number ${\rm h}(\mathcal{D})$, and can see that    $\mbox{e.d.}(G/B)={\rm h}(\mathcal{D})-1$ if and only if $G$ is of either classical Lie type or type G.

 \begin{table}[h]
  \caption{Effective  good divisibility of $G/B$ of type $\mathcal{D}$}\label{tabed22}
 { \upshape

\begin{tabular}{|c||c|c|c|c|p{0.7cm} |p{0.7cm}|p{0.7cm}|p{0.7cm}|p{0.6cm}|}
  \hline \rule{0pt}{12pt}
  % after \\: \hline or \cline{col1-col2} \cline{col3-col4} ...
  $\mathcal{D}$ &$\mbox{A}_n$ & $\mbox{B}_n$ & $\mbox{C}_n$ & $\mbox{D}_{n+1}$ & \hspace{0.2cm}$\mbox{E}_6$ & \hspace{0.10cm} $\mbox{E}_7$ &  \hspace{0.2cm}$\mbox{E}_8$ & \hspace{0.1cm} $\mbox{F}_4$ &     $\mbox{G}_2$   \\ \hline  \rule{0pt}{12pt}
 $\mbox{e.d.}(G/B)$ & $n$ & $2n-1$ & $2n-1$ & $2n-1$ &\hspace{0.1cm} $12$ & \centering{ $19$ } & \centering{ $40$} &\centering{ $12$} &  $5$   \\ \hline \rule{0pt}{12pt}
   ${\rm h}(\mathcal{D})$ &\centering{ $n+1$}  & \centering{$2n$}  & \centering{$2n$}  &  {$2n$}    &\hspace{0.1cm} $12$ & \centering{ $18$} & \centering{$30$} & \centering{$12$} &  $6$  \\
  \hline
\end{tabular}
  }   \end{table}

% As we should have  mentioned earlier, independently in \cite{MOS22}, Mu\~{n}oz, Occhetta and Sol\'{a} Conde have also obtained the above  results for rational homogeneous varieties of    classical  type. Nevertheless, our approach is completely different from theirs, and our results cover all the cases of exceptional   type  as well.
\begin{remark}
In   \cite{MOS22}, Mu\~{n}oz, Occhetta and Sol\'{a} Conde have also obtained {\upshape $\mbox{e.d}(G/P)$} when $G$ is of classical type.
They started with  a complete flag manifold $G/B$. %Their arguments
They made use of a beautiful induction on the rank of $G$  and computer computations  for some rational homogeneous varieties with low rank automorphism groups, with key ingredients including  Stumbo's description \cite{Stum} of certain elements in the Weyl group. Then {\upshape $\mbox{e.d.}(G/P)$} was obtained, due to a minor gap between {\upshape $\mbox{e.d.}(G/B)$} and an obvious upper bound of {\upshape $\mbox{e.d.}(\mathcal{D}(m))$}.
They also checked   {\upshape $\mbox{e.d.}(G/B)$}    for  types $\mbox{G}_2$, $\mbox{F}_4$ and $\mbox{E}_6$, while the cases for types $\mbox{E}_7$ and $\mbox{E}_8$ were computationally out of reach.

We are taking a completely different approach by   starting with Grassmannians $G/P_m$ of classical   type, which parameterize  isotropic  vector subspaces  of $\mathbb{C}^N$.   For such Grassmannians, the Schubert varieties $X^u$ are not only labeled by elements $u$ in a subset of the Weyl group of $G$, but also equivalently by
   other  combinatorial objects, including $k$-strict partitions and index sets \cite{BKT} (called also Schubert symbols in \cite{Ravi}).
The Bruhat order among index sets becomes much easier to study, with the price of complicated   dimension counting.
 We achieve Theorem \ref{mainthm1} by carefully analyzing the cardinality of relevant combinatorial sets for classical  type, and by directly comparing the Bruhat order among Grassmannian permutations for exceptional  type  with the help of  computer computations. Then we carry out all  {\upshape $\mbox{e.d.}(G/P)$} in Theorem \ref{mainthm2} using a reduction of Bruhat orders as mentioned above.
\end{remark}

In any category, it is important to study the morphisms between objects. One question of interest is to find   sufficient conditions for a morphism to be constant.
The answer is  well-known for morphisms $\mathbb{P}^r \to \mathbb{P}^n=Gr(1, n+1)$ of (smooth) projective varieites, which must be constant whenever  $r > n$.
 The same statement with $\mathbb{P}^n$ replaced by general complex Grassmannian $ Gr(m, n+1)$  was obtained by  Tango \cite[Corollary 3.2]{Tango},
    by studying the pullback of Chern classes of tautological vector bundles.
With the same idea,  Naldi and Occhetta \cite{NaOc} extended the result to morphism $M\to Gr(m, n+1)$ from smooth complex projective variety $M$, with the concept of  effective good divisibility $\mbox{e.d.}(M)$.
\begin{thm}\label{mainthm3} Let $G/P$ be of classical Lie type, and $M$ be a connected complex projective manifold. If {\upshape $\mbox{e.d.}(M) > \mbox{e.d.}(G/P)$}, then
any morphism $M\to G/P$ is constant.
\end{thm}
\noindent In \cite{MOS22},  Mu\~{n}oz, Occhetta and Sol\'{a} Conde provided a slightly more general application by assuming $M$ to be a complex projective variety (together with a straightforward extension of   $\mbox{e.d.}(M)$ in terms of nonvaninshing of products in the Chow ring of $M$).
 On the one hand, we reduce Theorem \ref{mainthm3} to the case when $G/P$ is a Grassmannian as in Proposition \ref{morphismprop}. Our proof  is mainly a straightforward generalization of that for $M\to Gr(m, n+1)$ in \cite{Tango, MOS20, NaOc}. It
  can   be directly extended to  the same setting for $M$ in \cite{MOS22} (see  {Remark } \ref{rmkmorphism}), and the two approaches are different.
 On the other hand, it seems difficult to find nice vector bundles over an even-dimensional quadric with expected properties. The proof for the exception is due to \cite[Corollary 5.3]{MOS22}.
As a   corollary of Theorems \ref{mainthm1}, \ref{mainthm2} and \ref{mainthm3}, the application \cite[Corollary 1.4]{MOS22} can be extended to $G$ of arbitrary type  as follows.
Therein the  hypothesis parabolic $Q$ in type $\mbox{E}_7$ (resp. $\mbox{E}_8$) is equivalent to $\emptyset\neq \Delta\setminus \Delta_Q \neq \{\alpha_7\}$ (resp. $\emptyset\neq \Delta\setminus \Delta_Q \not\subseteq \{\alpha_7, \alpha_8\}$).
We expect the additional assumption in types  $\mbox{E}_7$ and  $\mbox{E}_8$ to be removable with a better understanding of the relevant rational homogeneous varieties.
\begin{cor}\label{corsubdiag}
   Let   $P, \bar P,  Q$ be proper parabolic subgroups of $G$  with $Q$ containing   $\bar P$. Then
    there does   not exist any non-constant morphism $G/P\to Q/\bar P$, provided
      $Q\neq P_7$ if $G$ is of type  {\upshape $\mbox{E}_7$}, or  $Q\not \supseteq P_7\cap P_8$ if $G$ is of type {\upshape $\mbox{E}_8$}.
\end{cor}
\noindent The Dynkin diagram of  $\Delta_Q$ consists of $r$ connected components of type $\mathcal{D}^{(1)}, \cdots, \mathcal{D}^{(r)}$ respectively.
Then  $Q/\bar P$ is isomorphic to  a product $X_1\times \cdots \times X_r$ with $X_i$ being  a  rational homogeneous variety (possibly a point)   of type $\mathcal{D}^{(i)}$.
 The  hypothesis in the above corollary   is equivalent to that all $\mathcal{D}^{(i)}$'s are of classical type.

The paper is organized as follows. In section 2, we review basic facts on rational homogeneous varieties and prove Theorem \ref{mainthm2}. In section 3, we  prove Theorem \ref{mainthm1} for  Grassmannians of classical  type. In section 4, we provide the computational results for Grassmannians of   exceptional   type.  Finally in section 5, we prove Theorem \ref{mainthm3} as an application.

\subsection*{Acknowledgements}
The authors thank Pierre-Emmanuel Chaput, Haibao Duan,   Jianxun Hu, Xiaowen Hu, Hua-Zhong Ke, Naichung Conan Leung,  Leonardo Constantin Mihalcea, Lei Song and Heng Xie   for helpful discussions. The authors are extremely grateful to the  anomynous referee for the very careful reading  and the quite valuable comments.  C. Li  is  supported by NSFC Grants 12271529 and 11831017.

\section{Reduction of effective good divisibility of $G/P$}
The notion of \textit{effective  good divisibility}   was introduced in \cite{MOS20}, refining the notion of good divisibility \cite{Pan}. In this section, we reduce the effective good divisibility of a rational homogeneous variety to that of Grassmannians. Namely we prove Theorem \ref{mainthm2} first, which is completely independent of Theorem \ref{mainthm1}.

\subsection{Rational homogeneous varieties} We review some well-known facts on rational homogeneous varieties, and refer to \cite{Hump75, Brion} for more details.

Let $G$ be a connected simple complex Lie group,  $B$ a Borel subgroup of $G$  and  $B^-$ the Borel subgroup  opposite to $B$. Then $T=B\cap B^-$ is a maximal torus.
 Let $R$ denote the root system of $G$  determined by $(G, T)$, with a base  $\Delta$ of simple roots and the subset $R^+\supset \Delta$ of  positive
roots.
Then parabolic subgroups $P$
 of $G$ that contain $B$ are in one-to-one correspondence with the subsets $\Delta_P\subset \Delta$.
 The quotient $G/P$ is a rational homogenous variety, called also a flag variety.
 We take the same numbering of the simple roots $\alpha_i$ as in the Dynkin diagram $\mathcal{D}$ on \cite[page 58]{Hump}.
In particular, $G/P_m$ is called a Grassmannian, when the parabolic subgroup  $P_m$ corresponds to the subset $\Delta_{P_m}=\Delta\setminus \{\alpha_m\}$.

The Weyl group  $W=N_G(T)/T$  of  $G$ is generated by  simple reflections $\{s_\alpha \mid \alpha\in \Delta\}$, where $N_G(T)$ is the normalizer of $T$ in $G$.
  Let $W_P$ be the Weyl subgroup  generated by the simple reflections $\{s_\alpha \mid \alpha\in \Delta_P\}$.
   There is a unique element of minimal length in each coset in $W/W_P$ with respect to the standard length function   $\ell: W\to \mathbb{Z}_{\geq 0}$.
    Such minimal length representatives  form a subset $W^P\subset W$ bijective to $W/W_P$. Then any element $w\in W$ has a unique factorization $w=w_1w_2$ with $w_1\in W^P$ and $w_2\in W_P$, implying $\ell(w)=\ell(w_1)+\ell(w_2)$. Let $w_0$ (resp. $w_P$) denote the (unique) longest element in $W$ (resp. $W_P$).

 There are  Bruhat decompositions of the flag variety $G/P$ into affine spaces: $G/P=\bigsqcup_{w\in W^P}BwP/P=\bigsqcup_{u\in W^P}B^-uP/P$.  The Zariski closures   $X^w = \overline{BwP/P}$ (resp. $X_u = \overline{B^-uP/P}$) are (opposite) Schubert varieties of dimension   $\ell(w)$ (resp. codimension  $\ell(u)$).
 It follows that
        $$H^*(G/P)=H^*(G/P, \mathbb{Z})=\bigoplus_{u\in W^P}\mathbb{Z}[X_u]$$
has a $\mathbb{Z}$-basis of Schubert cohomology classes $[X_u]\in H^{2\ell(u)}(G/P, \mathbb{Z})$ (the Poincar\'e dual of the homology class of $X_u$).
There is a partial order $\leq$ on $W^P$, called the Bruhat order, with $u\leq w$ defined by any one of the following equivalent properties
\begin{enumerate}
  \item[i)] $X^u\subseteq X^w$;
  \item[ii)] $X_u\supseteq X_w$;
  \item[iii)] a reduced decomposition of $u$ can be obtained from  some substring of some   reduced decomposition of   $w$.
\end{enumerate}

\noindent The (transversal) intersection $X^w\cap X_u$ is   an irreducible and reduced closed subvariety of dimension  $\ell(w)-\ell(u)$ if $u\leq w$, or an  empty set otherwise.
It follows that
  \begin{equation}\label{Schubertdual}
   [X^w]\cup [X_u]=[X^w\cap X_u],\qquad   \int_X [X^w]\cup [X_u]=\delta_{w, u}.
  \end{equation} For any $v\in W^P$, (opposite) Schubert varieties are related by $X_v=w_0X^w$ with  $w=w_0vw_P\in W^P$. This implies the following well-known fact.
\begin{prop}\label{propnonvanishing}
  For any $u, v\in W^P$, $[X_u]\cup [X_v]\neq 0$ if and only if $u\leq w_0vw_P$.
\end{prop}

\subsection*{Quantum cohomology}
Here we provide a   brief description of the small quantum cohomology ring $QH^*(G/P)$  for self-containedness, and refer to \cite{FuPa, FuWo} for more details.

For any   variety $Y$, we denote by  $[Y]_h$ the homology class of $Y$.   Let
$\overline{\mathcal{M}}_{0,3}(G/P, \mathbf{d})$ denote the Kontsevich moduli space of all $3$-pointed
stable maps $(f : C \to G/P, p_1, p_2, p_3)$ of arithmetic genus zero and degree
$f_*([C]_h) = \mathbf{d}\in H_2(G/P, \mathbb{Z})$, which is a projective variety of dimension
$$\dim\overline{\mathcal{M}}_{0,3}(G/P, \mathbf{d})=\dim G/P+\langle c_1(G/P), \mathbf{d}\rangle.$$
 This space is equipped with evaluation maps
$\mbox{ev}_i : \overline{\mathcal{M}}_{0,3}(G/P, \mathbf{d}) \to G/P$ that send a stable map to the image $f(p_i)$ of the $i$-th marked point.
 For $\Delta\setminus\Delta_P=\{\alpha_{j_1}, \cdots, \alpha_{j_r}\}$, we simply denote $q_i=q_{\alpha_{j_i}}$. The quantum product in the quantum cohomology ring $QH^*(G/P)=(H^*(G/P)\otimes \mathbb{Z}[q_1, \cdots, q_r], \star)$ is defined by
   $$ [X_u]\star [X_v] =\sum_{\mathbf{d}\in H_2(G/P, \mathbb{Z})}\sum_{w\in W^P} N_{u, v}^{w, \mathbf{d}} [X_w]\mathbf{q}^{\mathbf{d}}$$
  where $\mathbf{q}^{\mathbf{d}}=\prod_{i=1}^rq_i^{d_i}$ for $\mathbf{d}=\sum_{i=1}^r d_i[X^{s_{\alpha_{j_i}}}]_h\in H_2(G/P, \mathbb{Z})$, and
   $$N_{u, v}^{w, \mathbf{d}}=\int_{\overline{\mathcal{M}}_{0,3}(X, \mathbf{d})}\mbox{ev}_1^*[X_u]\cup \mbox{ev}_2^*[X_v]\cup \mbox{ev}_3^*[X^w]$$
    is a genus zero, 3-pointed Gromov-Witten invariant, counting the number of holomorphic maps   in
    $\{\phi:\mathbb{P}^1\to G/P\mid  \phi_*([\mathbb{P}^1]_h)=\mathbf{d}; \phi(0)\in X_u, \phi(1)\in g_1 X_v, \phi(\infty)\in g_2 X^w\}$
    %$\phi:\mathbb{P}^1\to G/P$ of degree $\phi_*([\mathbb{P}^1]_h)=\mathbf{d}$ that satisfy $\phi(0)\in X_u, \phi(1)\in g_1 X_v$ and $\phi(\infty)\in g_2 X^w$
    for generic $g_1, g_2\in G$.
In particular,  $N_{u, v}^{w, \mathbf{d}}=0$ unless $\mathbf{d}$ is effective, namely $d_i\geq 0$ for all $i$. If $\mathbf{d}=\mathbf{0}$, then $\phi$ is a constant map, implying
      \begin{equation}
         [X_u]\star [X_v]|_{\mathbf{q}=\mathbf{0}}=\sum_{w\in W^P}N_{u, v}^{w, \mathbf{0}}[X_w]=[X_u]\cup [X_v].
      \end{equation}
Moreover,  the above sum   is indeed finite, due to the positivity of $c_1(G/P)$.

Let   $\langle\cdot, \cdot\rangle$ denote the natural pairing between cohomology classes and homology classes. The following proposition will be sufficient for our later use.
 \begin{prop}\label{quantumnonvanishi}
    The quantum cohomology   $QH^*(G/P)$ is a $\mathbb{Z}$-graded algebra with respect to the gradings $\deg[X_u]=\ell(u)$ and {\upshape $\deg q_{\alpha}=\langle c_1(G/P), [X^{s_{\alpha}}]_h\rangle$}, where $u\in W^P$ and   $\alpha\in \Delta\setminus \Delta_P$.
    Furthermore for any $v\in W^P$, $[X_u]\star [X_v]\neq 0$.
 \end{prop}
\noindent  The nonvanishing of the quantum product of Schubert classes is due to Fulton and Woodward \cite[Theorem 9.1]{FuWo}. The $\mathbb{Z}$-graded algebra structure of $QH^*(G/P)$ follows directly from the definition of quantum cohomology.

\subsection{Proof Theorem \ref{mainthm2}} Let us start with the notion of effective good divisibility introduced in \cite{MOS20}.
 \begin{defn}
{\upshape A complex projective manifold $M$ is said to have  \textit{effective
good divisibility  up to degree $s$}, denoted as  $\mbox{e.d.}(M) = s$,  if $s$ is
the maximum   integer such that $x_i\cup x_j\neq 0$ for any  effective classes $x_i\in H^{2i}(M)\setminus\{0\}$,
$x_j \in H^{2j}(M)\setminus\{0\}$ and any $i,j$ satisfying $i+j\leq s$.
  }
\end{defn}
Recall $\Delta\setminus\Delta_P=\{\alpha_{j_1}, \cdots, \alpha_{j_r}\}$. Denote by $\pi_i: G/P\to G/P_{j_i}$ the natural projection. For any $u\in W^P$, we denote by $u_i$
the minimal length representative  of the coset  $uW_{P_{j_i}}$ in $W^P/W_{P_{j_i}}$. Denote by $Z^{u_i}$ (resp. $Z_{u_i}$) the (opposite) Schubert varieties in $G/P_{j_i}$ of dimension (resp. codimension) $\ell(u_i)$. Denote by $d_{\alpha_{j_i}}(u, w)$ the smallest degree of a $T$-stable curve in $G/P_{j_i}$ connecting    the Schubert varieties
   $Z_{u_i}$ and $Z^{w_i}$. %(or equivalently the smallest degree of the quantum variable $\bar q_i$ occuring in the quantum product $[Z_{u_i}]\star[Z^{w_i}]$ in $QH^*(G/P_{j_i})=   H^*(G/P_{j_i})\otimes \mathbb{Z}[\bar q_i]$).
    The next proposition is Theorem 5 of \cite{BCLM}.
\begin{prop}\label{propcurvedeg}
   Let $u, w\in W^P$. There exists a stable curve of degree $\mathbf{d}=\sum_{i=1}^r d_i[X^{s_{j_i}}]_h$ connecting $X_u$ and $X^w$ if and only if $d_i\geq d_{\alpha_i}(u, w)$ for all $1\leq i\leq r$.
\end{prop}
 \noindent A $T$-stable curves of degree $\mathbf{0}$   is a $T$-fixed point of the natural action of $T$. There exists a stable curve of degree $\mathbf{0}$ connecting $X_u$ and $X^w$ if and only if   $X_u\cap X^w\neq \emptyset$, and hence it is equivalent to $u\leq w$.
 Similarly, $$0\geq d_{\alpha_{j_i}}(u, v)\geq 0 \Longleftrightarrow d_{\alpha_{j_i}}(u, v)=0 \Longleftrightarrow  Z_{u_i}\cap Z^{w_i}\neq \emptyset \Longleftrightarrow u_i\leq w_i.$$
 In other words, applying Proposition \ref{propcurvedeg} to the special case $\mathbf{d}=\mathbf{0}$, we obtain
  \begin{prop}\label{bruhatGP}
   For any $u, w\in W^P$, $u\leq w$ if and only if $u_i\leq w_i$ for all $i$.
 \end{prop}
 \noindent The above proposition is the key ingredient in our proof of Theorem \ref{mainthm2}, which reduces the Bruhat order on $W^P$ to that on   $W^{P_{j_i}}$.
This property has been well known for a long time for $G=SL(n+1, \mathbb{C})$ (see for instance exercise 8 on page 175 of \cite{Fult}).
We explain it for general $G$ from \cite[Theorem 5]{BCLM}, for not being aware of a precise reference elsewhere.

\bigskip

\begin{proof}[Proof of Theorem \ref{mainthm2}]
  The natural projection $\pi_i$ induces an injective homomorphism $\pi_i^*: H^*(G/P_{j_i})\to H^*(G/P)$, which sends a Schubert class $[Z_w]$ to the Schubert class $[X_w]$ of the same subscript. It follows that $\mbox{e.d.}(G/P)\leq \mbox{e.d.}(G/P_{j_i})$ for all $i$. Hence,
   $\mbox{e.d.}(G/P)\leq \min\{\mbox{e.d.}(G/P_{j_i})\mid 1\leq i\leq r\}$.

   Say $\mbox{e.d.}(G/P_{j_m})$ is the minimum among $\mbox{e.d.}(G/P_{j_i})$'s, $1\leq i\leq r$.
 Take any $u, v\in W^P$ with $\ell(u)+\ell(v)\leq \mbox{e.d.}(G/P_{j_m})$.
 Then $w:=w_0vw_P\in W^P$ and   we have the factorizations $u=u_i\bar u_i$, $w=w_i\bar w_i$, where $u_i, w_i\in W^{P_{j_i}}$ and $\bar u_i, \bar w_i\in W_{P_{j_i}}$.
 For any $\alpha\in \Delta_P$, we have $u(\alpha)\in R^+$ due to  $u\in W^P$. Since $\Delta_P\subseteq \Delta_{P_{j_m}}$,   $\bar u_i(\alpha)$ belongs to the root subsystem $R_{P_{j_m}}$. Then it follows from   $u_i\in W^{P_{j_m}}$ and $u_i(\bar u_i(\alpha))=u(\alpha)\in R^+$  that $\bar u_i(\alpha)$ is again a positive root. Therefore
 $\bar u_i\in W_{P_{j_i}}\cap W^P$. Similarly, we have  $\bar w_i\in W_{P_{j_i}}\cap W^P$.
Moreover, $\ell(u_i)\leq \ell(u)$, $[X_v]=[X^{w}]\in H^*(G/P)$, $[Z^{w_i}]=[Z_{w_0w_iw_{P_{j_i}}}]\in H^*(G/P_{j_i}) \mbox{ and}$
 \begin{align*}
    \ell(w_0w_iw_{P_{j_i}})&=\ell(w_0)-\ell(w_i)-\ell(w_{P_{j_i}})\\
    &=\ell(w_0)-\ell(w)+\ell(\bar w_i)- \ell(w_{P_{j_i}})\\
   &=\ell(w_0)-(\ell(w_0)-\ell(v)-\ell(w_{P}))+\ell(\bar w_i)- \ell(w_{P_{j_i}})\\
   &=\ell(v)+\ell(\bar w_iw_P)-\ell(w_{P_{j_i}})\\
   &\leq \ell(v).
 \end{align*}
 Therefore $\ell(u_i)+\ell(w_0w_iw_{P_{j_i}})\leq \ell(u)+\ell(v)\leq \mbox{e.d.}(G/P_{j_m})\leq \mbox{e.d.}(G/P_{j_i})$,
  implying $[Z_{u_i}]\cup [Z^{w_i}]=[Z_{u_i}]\cup [Z_{w_0w_iw_{P_{j_i}}}]\neq 0$. Thus $u_i\leq w_i$ for all $i$. By Proposition \ref{bruhatGP}, we have $u\leq w$.
  Hence, $[X_u]\cup [X_v]=[X_u]\cup [X^w]\neq 0$. It follows that $\mbox{e.d.}(G/P)\geq  \mbox{e.d.}(G/P_{j_m})=  \min\{\mbox{e.d.}(G/P_{j_i})\mid 1\leq i\leq r\}$.

  Hence, $\mbox{e.d.}(G/P)= \min\{\mbox{e.d.}(G/P_{j_i})\mid 1\leq i\leq r\}$.
\end{proof}

\section{Effective good divisibility of Grassmannians of classical type}\label{S:grassmannians}
In this section, we study $\mbox{e.d.}(G/P_m)$ for $G$ of classical type. The key technical propositions are Proposition \ref{thmtypeB} and Proposition \ref{thmtypeD}, dealing with the case of Lie types $\mbox{B}_n$ and $\mbox{D}_{n+1}$ respectively.
  .
\subsection{Isotropic Grassmannians} We refer to   \cite{BKT, Ravi} for more details on the facts reviewed here.
 A Grassmannian $G/P_m$ of type  $\mbox{A}_n$, i.e. when  $G=SL(n+1, \mathbb{C})$, is known as a complex Grassmannian $Gr(m, n+1)=\{\Sigma\leqslant \mathbb{C}^{n+1}\mid \dim \Sigma= m\}$. This interpretation can be generalized to other classical types.

Let $V$ be an $N$-dimensional  complex vector space  equipped with a non-degenerate skew-symmetric or symmetric  bilinear form $\omega(\cdot, \cdot)$, and denote
\[
X=IG_{\omega}(m, N) := \{\Sigma \leqslant V: \dim_\C \Sigma=m,\,\, \omega(\+v,\+w) = 0 \hphantom{.} \forall \+v,\+w \in \Sigma\}.
\]
$X$ is  smooth projective variety if $m\leq {N\over 2}$,  or an empty set otherwise. We usually use a different notation of $X$ in each specified type.
  \begin{enumerate}
    \item[b)] $\omega$ is symmetric and $N=2n+1$. Then $X=OG(m, 2n+1)$ is called an odd orthogonal Grassmannian. Set $k:=n-m$.
      \item[c)] $\omega$ is skew-symmetric, implying $N=2n$. Then $X=SG(m, 2n)$ is called a symplectic Grassmannian. Set $k:=n-m$.
     \item[d)] $\omega$ is symmetric and $N=2n+2$. Then $X=OG(m, 2n+2)$ is called an even orthogonal Grassmannian. Set $k:=n+1-m$.

   \end{enumerate}
 When $G$ is of types $\mbox{B}_n$, $\mbox{C}_n$  or $\mbox{D}_{n+1}$, $G/P_{m}$ can be realized as an isotropic Grassmannian $IG_\omega(m, N)$ as above,  with an additional requirement $m<n$ when   $G$ is of type $\mbox{D}_{n+1}$. The isotropic Grassmannian  $OG(n+1, 2n+2)$  consists of two connected components, which are isomorphic to  $G/P_{n}$ and $G/P_{n+1}$ of type  $\mbox{D}_{n+1}$  respectively. Moreover,   they are all isomorphic   to $OG(n, 2n+1)$.
 Therefore we will further assume $m<n$ in case d) above.

A partition is a weakly decreasing sequence of non-negative integers $\lambda=(\lambda_1\geq \cdots\geq \lambda_m\geq 0)$. The weight of $\lambda$ is the sum $|\lambda|=\sum_i \lambda_i$.
For  $G=SL(n+1, \mathbb{C})$, the Weyl group $W$ is the group $S_{n+1}$ of permutations of $(n+1)$ objects. The subset $W^{P_m}$ is bijective to the set $\mathcal{P}_{m, n+1}$  of partitions inside an $m\times (n+1-m)$ rectangle, by $w\mapsto (w(m)-m, \cdots, w(1)-1)$.  The length $\ell(u)$ equals  the weight of the corresponding partition, and the Bruhat order on $W^{P_m}$ coincides with the standard partial order on $\mathcal{P}_{m, n+1}$  as real vectors.
In  \cite{BKT},  Buch, Kresch and Tamvakis introduced the following  notion of $k$-strict partitions for $G/P_m$ of other classical types,
  making the codimension of Schubert varieties apparent.

\begin{defn}
  {\upshape A partition $\lambda$ is called \textit{$k$-strict} if no part greater than $k$ is repeated, namely $\lambda_j>k \Rightarrow \lambda_j>\lambda_{j+1}$. }
\end{defn}
\noindent Notice that the cases b), c) and d) can be distinguished by the triple $(N, n, k)$. There is a bijection $u\mapsto \lambda(u)$ between  $W^{P_m}$ and    $\mathcal{P}_N(k, n)$, the set of $k$-strict partitions contained in an $m\times (n+k)$ rectangle in cases $\mbox{B}_n, \mbox{C}_n$ (for $\mbox{D}_n$ one should use a slight modification
 of $k$-strict partitions in order to describe Schubert varieties; we will recall this modification in section \ref{secevenorth}).
Schubert varieties $X_u=X_{\lambda(u)}$ have apparent codimension $|\lambda(u)|$, yet the Bruhat order $u\leq w$ becomes complicated in terms of $k$-strict partitions in general.

  It is a central problem to find a manifestly positive formula of all Schubert structure constants $N_{u, v}^{w, 0}$ for $H^*(G/P)$.
   This remains unknown even for $G/P_m$ in cases b), c), d) with $2\leq m<n$.
  The case $m=n$ has been much better understood (we refer to \cite{ThYo} and references therein).
There is an exact sequence of vector bundles over $X$,
  \begin{equation}\label{bundleseq}
     0\longrightarrow \mathcal{S}\longrightarrow \mathcal{V}_X\longrightarrow \mathcal{Q}\longrightarrow 0
  \end{equation}
   with $\mathcal{S}$ the tautological subbundle of the trivial bundle $\mathcal{V}_X=X\times V$, whose fiber at a point $\Sigma\in X$ is just the vector space $\Sigma$.
   The multiplication of $c_i(\mathcal{S}^\vee)$ (resp. $c_j(\mathcal{Q})$) by general Schubert classes are known as Pieri rules and are provided in \cite{PrRa96, PrRa03} (resp. \cite{BKT}).
    The Chern classes of $\mathcal{Q}$ and the dual vector bundle $\mathcal{S}^\vee$ are Schubert classes labeled by special $k$-strict partitions (up to a scalar of $2$).
In particular, $\mathcal{S}^\vee$ is of rank $m$, $\mathcal{Q}$ is of rank $N-m$, and for $X=OG(m, N)$ (with $m<n$) we have
  $$c_m(\mathcal{S}^\vee)=[X_{(1,\cdots, 1)}],\qquad  c_{n+k}(\mathcal{Q})= 2[X_{(n+k, 0, \cdots, 0)}],\qquad c_{n+k+1}(\mathcal{Q})=0,$$
for more detailed explanation of which we refer to \cite{BKT, Tamv}. It follows that
   \begin{equation}\label{Chernprod}
       0=c_{n+k}(\mathcal{Q})\cup c_m(\mathcal{S}^\vee)=2[X_{(n+k, 0, \cdots, 0)}]\cup [X_{(1,\cdots, 1)}] \in H^{2(n+k+m)}(X).
   \end{equation}
We remark that when $X$ is a complex or symplectic Grassmannian, the scalar $2$ does not occur, and we have $c_m(\mathcal{S}^\vee)\cup c_{n+1-m}(\mathcal{Q})=0$ for $Gr(m, n+1)$.

The Borel subgroup $B$ of $G$ stabilize an isotropic complete flag $F_\bullet$, which specifies a  basis $\{\mathbf{e}_1, \cdots, \mathbf{e}_N\}$ of $V$  that satisfies $(\mathbf{e}_i, \mathbf{e}_j)=\delta_{i+j, N+1}$ for all $i\leq j$; moreover,  $F_j = \langle \+e_1,\ldots,\+e_j \rangle$, the span of the first $j$ basis vectors.
    Schubert varieties can also be parameterized by  \textit{index sets} $\mathscr{P}$ (called also Schubert symbol in \cite{Ravi}), which are subsequences $\mathscr{P}=\{p_1<p_2<\cdots<p_m\}$ in $[N]$ that
      satisfy $p_i+p_j\neq N+1$ for all $i\neq j$.  There is a bijection between
    $\mathcal{P}_N(k, n)$ and the set $\mathfrak{S}(X)$ of index sets.
    The set $\mathfrak{S}(X)$ is equipped with a standard partial order, say $\mathscr{P}\leq \hat{\mathscr{P}}$ if and only if $p_i\leq \hat p_i$ for all $1\leq i\leq m$.
    The parameterizations by index sets provide geometric descriptions of Schubert varieties and make the Bruhar order apparent in terms of the standard order $\leq$, while it pays the price at  dimension counting of Schubert varieties.
    In next two subsections, we will restrict to $X=OG(m, N)$, with $1\leq m<n$ in case b) and $2\leq m<n$ in case d), and
    will describe the correspondence $\lambda\mapsto \mathscr{P}(\lambda)$  precisely. We will deal with the remaining cases in section \ref{Pfclassical}.
    We will interchange the parameterizations $u$, $\lambda=\lambda(u)$ and $\mathscr{P}=\mathscr{P}(\lambda)$.
\subsection{Odd orthogonal Grassmannians}
  Recall $k=n-m$ for $X=OG(m, 2n+1)$. The bijection $\Phi$ is given by
  \begin{align*}
     \Phi:& \mathcal{P}_{2n+1}(k, n)\longrightarrow \mathfrak{S}(X)=\{\mathscr{P}\subset [1, 2n+1]\mid p_i+p_j\neq 2n+2,\,\, \forall i\neq j;\,\,\,   n+1\not\in \mathscr{P}\}\\
          &\qquad \lambda\mapsto \mathscr{P}(\lambda)=(p_1(\lambda),\cdots, p_m(\lambda))\qquad\quad\mbox{with}
  \end{align*}
   \begin{equation} \label{pjformula}
      p_j(\lambda)=n+k+1-\lambda_j+ \#\{i<j:\lambda_i+\lambda_j\leq 2k+j-i\}+\begin{cases}
         1,&\mbox{if } \lambda_j\leq k,\\0,&\mbox{if }\lambda_j>k.
      \end{cases}.
   \end{equation}
  The Schubert variety $X_\lambda=X_{\mathscr{P}(\lambda)}$   (relative to the isotropic  flag $F_\bullet$) is given by
         $$X_\lambda=X_\lambda(F_\bullet)=\{\Sigma \in X\mid \dim (\Sigma\cap F_{p_j(\lambda)})\geq j, \,\,\forall 1\leq j\leq m\}=:X_{\mathscr{P}(\lambda)},$$
  where the rank condition   $\dim (\Sigma\cap F_{p_j(\lambda)})\geq j$ becomes trivial whenever $\lambda_j=0$.
 The next property follows immediately from the above description of Schubert varieties.
  \begin{prop} For any $\mathscr{P}, \mathscr{Q}\in \mathfrak{S}(X)$, $X_{\mathscr{P}}\subseteq X_{\mathscr{Q}}$ if and only if $\mathscr{P}\leq \mathscr{Q}$.
 \end{prop}
\noindent In particular we define $\mathscr{P}\preceq  \mathscr{Q}$ if and only if $X_{\mathscr{P}}\subseteq X_{\mathscr{Q}}$ as in \cite{Ravi}.
 Then
  \begin{equation}\label{bruhatequivB}
     \mu\preceq \lambda \Longleftrightarrow \mathscr{P}(\lambda)\preceq  \sP(\mu)=:\mathscr{Q} \Longleftrightarrow \sP\leq \sQ.
  \end{equation}
The dual  $\lambda^\vee$ of $\lambda=\lambda(u)$ is the $k$-strict partition in $\mathcal{P}_{2n+1}(k, n)$ that corresponds to $w_0uw_{P_m}$.
By \eqref{Schubertdual} and  $[X^u]=[X_{w_0uw_{P_m}}]$, $\lambda^\vee$ is the unique element  satisfying
    $$\int_X [X_\mu]\cup [X_{\lambda^\vee}]=\delta_{\mu, \lambda}, \forall\,\,  \mu\in \mathcal{P}_{2n+1}(k, n)$$
The dual index set $\mathscr{P}^\vee=\sP(\lambda)^\vee:=\sP(\lambda^\vee)$  has a simple description by
    \begin{equation}\label{pjdualformulaB}
       p_j^\vee=2n+2-p_{m+1-j},\qquad \forall 1\leq j\leq m.
    \end{equation}
% As from \cite[section 4]{BKT} (see also \cite[Lemma 2.1]{Ravi}), the Bruhat order $\lambda \preceq \mu$ can be described by the standard partial order for index sets as follows.

We introduce the following definition for both $OG(m, 2n+1)$ and $OG(m, 2n+2)$.

\begin{defn}\label{defnsmall}
  {\upshape  Let $\lambda =(\lambda _1, \lambda _2, \cdots, \lambda _m)$ be a $k$-strict partition, and $1\leq i\textless j\leq m$. We call the subscript pair $(i,j)$   \textit{small} if $\lambda _i+\lambda _j< 2k+1+j-i$, or \textit{big} otherwise. We call the $k$-partition $\lambda$    \textit{small} if all subscript pairs of $\lambda$ are small,  or \textit{big} otherwise. {We always denote $\lambda_0:=n+k+1$ for convention, so that    $(0,j)$ is   treated as a big subscript pair regardless of whether $\lambda$ is small.}
}\end{defn}

%The next proposition  follows immediately from the above definition. For the sake of simplicity, we only proof $ii)$.
\begin{lemma}\label{propbig}
 Let $\lambda =(\lambda _1, \cdots,  \lambda _m)$ be a $k$-strict partition,   and $\mathscr{P}=\mathscr{P}(\lambda)=\{ p_1\textless p_2\textless \cdots \textless p_m\}$.
 Suppose $1\leq i<j\leq m$.

\begin{enumerate}
  \item[i)]  If $(i, j)$ is big, so is   $(i-1,j)$. Moreover, if   $i<j-1$, then $(i, j-1)$ is  big.
   \item[ii)] If $\lambda$ is big, then there exists an integer $1< a\leq m$, such that for any $1\textless i\leq a$,  $(1,i)$ is big, and that for any $a\textless j\leq m$,  $(1,j)$ is not big.

% If $\lambda$ is big, then there exists an integer $1< a\leq m$, such that    $(1,i)$ is big if   $1\textless i\leq a$, or not big otherwise.
  \item[iii)] For any $1\leq j\leq m$, there exists a unique integer $0\leq a < j$, such that for any $0\leq i\leq a$, $(i,j)$ is big, and that  for any $a+1\leq t\textless j$, $(t,j)$ is not big.
\end{enumerate}
\end{lemma}

\begin{proof} Statement i)  follows immediately from  Definition \ref{defnsmall}.

If $\lambda$ is big, then there exists a big pair $(p,q)$ where $p<q$. By using the first half of i) repeatedly,
  we conclude that  $(1,q)$ is big. Set $a=max\{b|(1,b)\text{ is big}\}$. Then $(1,j)$ is small for any $a\textless j\leq m$.
By the second half of i),
   $(1,i)$ is big for all $1<i\leq a$.       Hence, statement ii) holds.

   The argument  for statement iii) is similar to that for statement ii).
\end{proof}

Due to  formula \eqref{pjformula}, we define a map $f: \mathcal{P}_{N}(k)\times [m] \to [0, m-1]$ by letting $f(\lambda, j)$ count the number of big subscript pairs of the form $(i, j)$. Namely,
\begin{equation} \label{defflambda}
     f(\lambda ,j):=j-1- \#\{i<j:\lambda_i+\lambda_j\leq N-2m-1+j-i\}.
\end{equation}
 Notice $N-2m-1=2k$ for $N=2n+1$.
In particular, we have $f(\lambda ,j)\leq j-1$ for any $j$, and $f(\lambda ,j)=0$ if $m=1$. Moreover, the value $f(\lambda ,j)$ coincides with the unique number $a$ in  Lemma \ref{propbig} iii).
We will prove a result similar to Proposition \ref{thmtypeB} also for $OG(m,2n+2)$.

\begin{prop}\label{thmtypeB}
Let $\lambda, \mu$ be $k$-strict partitions. If $\left| \lambda \right| +\left| \mu \right|< 2n, $ then for all $j$, the inequality
$f(\lambda ,m+1-j)+f(\mu ,j)\leq n+k-\lambda _{m+1-j}-\mu _j$ holds.
%$\lambda _{m+1-j}+\mu _j\leq n+k-f(\lambda ,m+1-j)-f(\mu ,j)$ holds.
\end{prop}

\begin{proof}
If $m=1$, then $f$ is a constant map with image $0$, and hence the expected inequality becomes trivial by noting $k=n-1$. Now we consider  $m\geq 2$, and give the proof by contradiction.

Assume $f(\lambda ,m+1-j)+f(\mu ,j)> n+k-\lambda _{m+1-j}-\mu _j$ for some $j$. Then it directly follows from the definition of the map $f$ that
$$\lambda _{m+1-j}+\mu _j> n+k-f(\lambda ,m+1-j)-f(\mu ,j)\geq n+k-(m-j)-(j-1)=2k+1.$$
%If $j=1$, then we can get
%$|\lambda|+|\mu|=\sum_{i=1}^{m}\lambda_i+\mu_1> m\lambda_m+(2k+1-\lambda_m)$, which is greater than $2n$ and gives a contradiction when $\lambda\geq 2$. If $\lambda_m\leq 1$, since $j=1$ we have $f(\mu,1)=0$, so that $f(\lambda,m)>n+k-\lambda_m-\mu_1$
Moreover, either $\lambda_{m+1-j}$ or $\mu_j$ must be less than $k+1$ (otherwise, we would deduce a contradiction from the inequalities
  $\left| \lambda \right| +\left| \mu \right|\geq \sum^{m+1-j}_{i=1}\lambda _i+\sum^{j}_{i=1}\mu _i\geq 2k+2+\sum^{m-j}_{i=1}(k+2)+\sum^{j-1}_{i=1}(k+2) =2n+(m-1)k\geq 2n$).
Without loss of generality, we assume $\lambda _{m+1-j}\leq k$, which implies $\mu _j\geq k+2$ and hence $\mu _{j-1}\textgreater \mu _j$ by the definition of $k$-strict partitions.
Here we  use the above convention $\mu_0=n+k+1$ when $j=1$.
Thus $\mu _{j-1}+\mu _j\geq 2k+5\textgreater 2k+1+j-(j-1)$. That is,  $(j-1,j)$ is a big subscript pair.  Hence, all subscript pairs $(i, j)$ are big  by Lemma \ref{propbig} iii), implying $f(\mu ,j)=j-1$. Consequently, we denote $a:=f(\lambda, m+1-j)$ and have
   $\lambda _{m+1-j}+\mu _j>n+k-a-j+1.$
Without loss of generality, we assume
\begin{equation}\label{proofBeq}
   \lambda _{m+1-j}+\mu _j=n+k-a-j+2.
\end{equation}
(Otherwise,  $\lambda _{m+1-j}+\mu _j>n+k-a-(j-1)+1$. Let $\tilde \mu=(\mu_1-1,\cdots, \mu_j-1,$  $\mu_{j+1},\cdots, \mu_m)$.
   It still follows from $\tilde \mu_j=\mu_j-1\geq k+1$ that $f(\tilde \mu, j)=j-1=f(\mu, j)$,
 and hence $a+f(\tilde \mu ,j)> n+k-\lambda _{m+1-j}-\tilde \mu _j$. Hence, we can replace $\mu$ by $\tilde \mu$.)

Notice  that $a$ is  the unique number in Lemma \ref{propbig} iii) with respect to $m+1-j$. Thus for any $i\leq a$,
  $(i, m+1-j)$ is a big subscript pair, namely
   \begin{equation}
     \lambda _{i}+\lambda _{m+1-j}\geq 2k+1+m+1-j-i=N-m+1-j-i.
   \end{equation}
Hence,
\begin{align*}
     & \left| \lambda \right| +\left| \mu \right|\\
     \geq &\sum^{a}_{i=1}\lambda _i+\sum^{m+1-j}_{i=a+1}\lambda _i+\sum^{j}_{i=1}\mu _i\\
     \geq &\sum^{a}_{i=1}(2k+2+m-j-i-\lambda _{m+1-j})
     +(m-j-a+1)\lambda _{m+1-j}+j\mu _j+\dfrac{j(j-1)}{2}\\
    = &(k+2+n-j)a-\dfrac{a(a+1)}{2} +(n-k-j-2a+1)(n+k-a-j+2-\mu _j)\\
     &\qquad +j\mu _j+j(j-1)/2\\
    = &(2a+2j-m-1)\mu_j+L_0(j, a)\\
   =&I(j, a, \mu_j)
   \end{align*}
where the functions $L: \mathbb{R}^2\to \mathbb{R}$ and $I: \mathbb{R}^3\to \mathbb{R}$ are defined by
\begin{align*}
  L_0(x, y)&:={3\over 2}y^2+(2x-2n-{7\over 2})y+ {3\over 2}y^2-(2n+{7\over 2})x+n^2+3n-k^2-k+2,\\
  I(x, y, z)&:=(2x+2y-m-1)z+ L_0(x, y).
\end{align*}
  We want to show $I(j, a, \mu_j)\geq 2n$ for all $j\in [m]$. We notice  $k+2\leq \mu_j \leq n+k-j+1$, and $I(x, y, z)$ is linear in $z$.

\begin{enumerate}
   \item[(i)] Assume  $2a+2j-m-1\geq 0$. Then we have
     $$I(j, a, \mu_j)\geq I(j, a, k+2)={3\over 2}a^2-(2n-2k-2j-{1\over 2})a+L_1(j)$$
      with $L_1(j)$ depending only on $j$. As a quadratic function in $y$, $I(j, y, k+2)$ takes minimum value at $y= {1\over 3}(2n-2k-2j-{1\over 2})$. While $a$ is an integer,
       $I(j, a, k+2)$ takes minimum value at the integer most close to ${1\over 3}(2n-2k-2j-{1\over 2})$, or equivalently most close to  ${1\over 3}(2n-2k-2j)$. Therefore,
         $$I(j, a, \mu_j)\geq I(j, {1\over 3}(2n-2k-2j), k+2)={5\over 6}j^2-{1\over 6}(4n-4k-1)j+L_2(n, k)$$
      with $L_2(n, k)$ independent of $j$. Since $j$ is an integer,    $I(j, {1\over 3}(2n-2k-2j), k+2)$ takes minimum value at the integer most close to ${1\over 5}(2n-2k-1)$, which is most close to ${1\over 5}(2n-2k)$ as well. Noting $k=n-m$ and by direct calculations, we have
     $$I(j, a, \mu_j)\geq I({1\over 5}(2n-2k), {1\over 3}(2n-2k-2j), k+2)=2n+n(m-1)-{4\over 5}m^2+{2\over 5}m.$$ Hence,
       $I(j, a, \mu_j) \geq 2n$, following from
       \begin{equation}
        \label{ineqBm}  n(m-1)-{4\over 5}m^2+{2\over 5}m\geq 0\qquad\mbox{due to } n>m\geq 2.
       \end{equation}
   \item[(ii)]   Assume  $2a+2j-m-1\geq 0$. Then
         we have $$I(j, a, \mu_j)\geq I(j, a, n+k-j+1)= {3\over 2}a^2+(2k-{3\over 2})a+L_3(j).$$
    Since $a=f(\lambda, m+1-j)\in [0, m-j]$ and   $2k-\dfrac{3}{2}\textgreater 0$ for $m<n$,
      we have
       \begin{align}
        \label{typeBpos}  I(j, a, \mu_j)\geq I(j, 0, n+k-j+1)
          &=-\dfrac{j^2}{2} +(n+k-\dfrac{1}{2})j+(n-k+1)\\
           &\geq  I(1, 0, n+k-1+1)=2n.
       \end{align}
    Here the second inequality holds by noting the function $I(j, 0, n+k-j+1)$ in $j$ is  increasing on $[1, m]$.
\end{enumerate}
In either cases, we deduce the contradiction    $2n>\left| \lambda \right| +\left| \mu \right|\geq 2n$.
\end{proof}
\begin{remark}
  The above arguments are also valid for the case $m=n\geq 3$, except that the reason ``$2k-{3\over 2}>0$ for $m<n$" for obtaining    \eqref{typeBpos} should be replaced by
  ``$0$ and $1$ are the integers most and equally  close to ${1\over 2}-{2\over 3} \cdot 0$".  For the case $m=n=2$,  the inequality \eqref{ineqBm} does not hold, while an individual  verification can be  made easily.
\end{remark}

\begin{cor}\label{cortypeB}
Let $\lambda, \mu$ be $k$-strict partitions. If $\left| \lambda \right| +\left| \mu \right|< 2n, $ then $\sP^{\vee}(\lambda)\preceq\sP(\mu)$ and $[X_\lambda]\cup [X_\mu]\neq 0$.
\end{cor}

\begin{proof}
Write $\sP=\mathscr{P}^\vee(\lambda)=\{p_1^\vee< \cdots<p_m^\vee\}$ and $\sQ=\sP(\mu)=\{q_1<\cdots<q_n\}$.
By the combination of the formulas   \eqref{pjformula}, \eqref{pjdualformulaB} and \eqref{defflambda}, $ \sP^{\vee}\leq \sQ$ holds if and only if for all $j\in [m]$, the following inequality holds.
 $$f(\lambda ,m+1-j)+f(\mu ,j)\leq n+k- \lambda _{m+1-j}-\mu_j
          -1+\begin{cases}
             1,&\mbox{if } \lambda_{m+1-j}\leq k\\
             0,&\mbox{if } \lambda_{m+1-j}> k
          \end{cases}+\begin{cases}
             1,&\mbox{if } \lambda_{\mu_j}\leq k\\
             0,&\mbox{if } \lambda_{\mu_j}> k
          \end{cases}.$$
 Since $\left| \lambda \right| +\left| \mu \right|< 2n$, either $\lambda_{m+1-j}\leq k$ or $\mu_{j}\leq k$ must hold (as from the second paragraph of the proof of Proposition \ref{thmtypeB}). It follows that $$0\leq -1+\begin{cases}
             1,&\mbox{if } \lambda_{m+1-j}\leq k\\
             0,&\mbox{if } \lambda_{m+1-j}> k
          \end{cases}+\begin{cases}
             1,&\mbox{if } \lambda_{\mu_j}\leq k\\
             0,&\mbox{if } \lambda_{\mu_j}> k
          \end{cases}.$$ Therefore  we have  $\sP^{\vee}(\lambda)\preceq\sP(\mu)$  by  Proposition \ref{thmtypeB} and the equivalences in   \eqref{bruhatequivB},
          and  consequently $[X_\lambda]\cup [X_\mu]\neq 0$ holds by Proposition \ref{propnonvanishing}.
    \end{proof}
   \subsection{Even orthogonal Grassmannians}\label{secevenorth}
Recall $k=n+1-m$ for $OG(m, 2n+2)$, and we restrict to $2\leq m<n$.
 Schubert varieties of $X=OG(m, 2n+2)$ are indexed by  elements in the set
   \begin{align*}
      \mathcal{P}_{2n+2}(k, n)&:=\{(\lambda, 0)\Big| {\lambda \mbox{ is a } k\mbox{-strict partition inside an } m\times (n+k) \mbox{ rectangle};\atop \lambda_j\neq k \mbox{ for all } j} \}\\
      &\quad\,\, \bigcup \{(\lambda, t)\Big| {\lambda \mbox{ is a } k\mbox{-strict partition inside an } m\times (n+k) \mbox{ rectangle};\atop
                                                      \lambda_j= k \mbox{ for some } j;\quad t\in \{1, 2\}} \}.\\
    \end{align*}
   This becomes quite  a bit more involved then odd orthogonal Grassmannians, due to the disconnectedness of $OG(n+1, 2n+2)$. There exists an alternate isotropic flag $\tilde{F}_\bullet$
   with  the properties (1) $\tilde F_{2n+2-i}^\bot=\tilde F_{i}=F_i$ for all $i\in [n]$, and (2)
  the connected component of $\tilde F_{n+1}^\bot=\tilde F_{n+1}$ in $OG(n+1, 2n+2)$ is distinct from that of $F_{n+1}$.
  Both isotropic flags are needed in defining   Schubert varieties with subscripts $(\lambda, 1)$ or $(\lambda, 2)$ by rank conditions. By \cite[Proposition 4.7]{BKT}, there is a bijection   $\Psi$ given by
  \begin{align*}
     \Psi:& \mathcal{P}_{2n+2}(k, n)\longrightarrow \mathfrak{S}(X)=\{\mathscr{P}\subset [1, 2n+2]\mid p_i+p_j\neq 2n+3,\,\, \forall i\neq j\}\\
          &\qquad (\lambda, t)\mapsto\mathscr{P}(\lambda, t)=(p_1(\lambda, t),\cdots, p_m(\lambda, t))\qquad\quad\mbox{with}
  \end{align*}
   \begin{equation} \label{pjformulaD}
         p_j=p_j(\lambda, t)=n+k-\lambda_j+\#\{i<j|\lambda_i+\lambda_j\leq N-2m-1+j-i\}+ g(\lambda, t, j)
    \end{equation}
   where the function $g: \mathcal{P}_{2n+2}(k, n)\times [m]\to \{1, 2\}$ is defined by
  \begin{equation}\label{funcgg}
   g(\lambda, t, j):=  \begin{cases}
        1,&\mbox{if } \lambda_j>k, \text{or } \lambda_j=k<\lambda_{j-1} \mbox{ and } n+j+t
            \text{ is even,}\\
        2,&\mbox{otherwise.}
        \end{cases}
  \end{equation}
The bijection $\Psi$  satisfies the following properties:
  \begin{equation}\label{eqntypepartD}
     \mbox{(i) }  \lambda_j\leq k\Longleftrightarrow p_j>n;\quad
 \mbox{(ii) } \lambda_j=k<\lambda_{j-1}\Longleftrightarrow p_j\in\{n+1,n+2\}
  \end{equation}

  The Schubert variety $X_\sP=X_{\mathscr{P}(\lambda, t)}(F_\bullet)=X_{(\lambda, t)}(F_\bullet)$     is of codimension $|\lambda|$, independent of the type $t$ of $\lambda$. We can simply denote $X_{\lambda}=X_{(\lambda, 0)}$ without confusion.
  If $n+2\notin \sP$, then
        $X_\sP=
            \{\Sigma \in X\mid \dim (\Sigma\cap F_{p_j})\geq j, \,\,\forall 1\leq j\leq m\}$, while if $n+2\in\sP$, then we have
         $$X_\sP=
             \{\Sigma \in X\mid \dim (\Sigma\cap F_{p_j})\geq j, \mbox{if } p_j\neq n+2;
                \dim (\Sigma\cap \tilde F_{n+1})\geq j, \mbox{if } p_j=n+2\} .$$
    For $\sP=\sP(\lambda, t)$, we define $\mathfrak{t}(\sP)=t$. As shown in \cite[section 4.3]{BKT}, we have
    \begin{prop}\label{bruhatD} For any $\mathscr{P}, \hat{\mathscr{P}}\in \mathfrak{S}(X)$, $X_{\mathscr{P}}\subseteq X_{\hat{\mathscr{P}}}$ if and only if both of the following hold.
\begin{enumerate}
\item $\sP \leq \hat{\sP}$;
%\item if there exists $c \in [n-1]$ such that $[n+1]\setminus [c] \subset [\sP] \cap [\sQ]$ and $\#\sP \cap [c] = \#\sQ \cap [c]$, then we have $\mathfrak{t}(\sP) = \mathfrak{t}(\sQ)$.
\item if $\hat p_i=n+2 \text{ for some } i$,   then $p_i\neq n+1$.
\end{enumerate}

 \end{prop}

With respect to the identification $\mathcal{P}_{2n+2}(k, n)\to W^{P_m}$, $(\lambda, t)\mapsto u=u(\lambda, t)$,
  the dual   $(\lambda, i)^\vee$ is the element in  $\mathcal{P}_{2n+2}(k, n)$ that corresponds to $w_0uw_{P_m}$.
 The dual index set $\mathscr{P}^\vee=\sP(\lambda, i)^\vee:=\sP((\lambda, i)^\vee)$ is given by (see \cite[Lemma 3.3]{Ravi})
    \begin{equation}\label{pjdualformulaD}
       p_j^\vee=\begin{cases}
          2n+3-p_{m+1-j},&\mbox{if } n \mbox{ is odd or } p_{m+1-j}\not\in \{n+1, n+2\},\\
          p_{m+1-j},&\mbox{if } n \mbox{ is even and }p_{m+1-j}\in \{n+1, n+2\}.
       \end{cases}
    \end{equation}

For any $(\lambda, t)\in \mathcal{P}_{2n+2}(k, n)$ and $j\in [m]$, we associate a number  $f(\lambda, j)$ as defined in \eqref{defflambda}, which is independent of the type $t$ of $\lambda$. Notice $N-2m+1=2k+1$ for $N=2n+2$.

\begin{prop}\label{thmtypeD}
    Let  $(\lambda, t_1)$ and $(\mu, t_2)$ be  in $\mathcal{P}_{2n+2}(k, n)$, and satisfy $|\lambda|+|\mu|<2n+1$.
   \begin{enumerate}
     \item For any $j\in [m]$,   $g(\lambda, t_1, m+1-j)+g(\mu, t_2, j)\geq 3$.
     \iffalse \begin{enumerate}
        \item $\lambda_{m+1-j}>k$ or $\lambda_{m+1-j}=k<\lambda_{m-j}$;
        \item $\mu_j>k$ or $\mu_j= k<\mu_{j-1}$.
      \end{enumerate}
      \fi
     \item For any $j\in [m]$, we have $f(\lambda ,m+1-j)+f(\mu ,j)\leq n+k-\lambda _{m+1-j}-\mu _j$.
   \end{enumerate}
\end{prop}

\begin{proof} We give the proof  by contradiction.

(1) Assume $g(\lambda, t_1, m+1-j)+g(\mu, t_2, j)<3$ for some $j$, then it follows from the definition in \eqref{funcgg} that
 $\lambda_{m+1-j}\geq k$ and $\mu_j\geq k$ both hold.
        Then we have
    \begin{equation*}
        \begin{split}
            |\lambda|+|\mu| \geq\sum_{i=1}^{m+1-j}\lambda_i+\sum_{i=1}^{j}\mu_j
            &\geq\sum_{i=0}^{m-j}(k+i)+\sum_{i=0}^{j-1}(k+i)\\
          %  &=(m+1)k+\frac{(m+1-j)(m-j)}{2}+\frac{j(j-1)}{2}\\
            & =(m+1)k+(j-{m+1\over 2})^2+ \frac{m^2-1}{4}\\
            & \geq (m+1)k+\frac{m^2-1}{4}\\
            &= 2n+1+n(m-1)-{3\over 4}m^2-{1\over 4}.
        \end{split}
    \end{equation*}
 The last equality holds by noting $k=n+1-m$. Since $2\leq m<n$,
    \begin{equation}\label{typeDmlarge}
      n(m-1)-{3\over 4}m^2-{1\over 4}\geq (m+1)(m-1) -{3\over 4}m^2-{1\over 4}={1\over 4}(m^2-5)\geq -{1\over 4}.
    \end{equation}
    This implies a contradiction  $2n+1>|\lambda|+|\mu|\geq 2n+1-{1\over 4}$, since  $|\lambda|+|\mu|$ is an integer.

(2)  Assume $f(\lambda ,m+1-j)+f(\mu ,j)> n+k-\lambda _{m+1-j}-\mu _j$ for some $j\in [m]$. Then
$$\lambda _{m+1-j}+\mu _j>n+k-(m-j)-(j-1)=2k$$
   First notice that either $ \lambda _{m+1-j}< k$ or $\mu_j< k$ must hold. Otherwise, they are both larger than or equal to $k$, and either of them must be larger than $k$, say $\mu_j$; then we would deduce the following contradiction.
  \begin{align*}
     2n+1>|\lambda|+|\mu|\geq\sum_{i=1}^{m+1-j}\lambda_i+\sum_{i=1}^{j}\mu_j &\geq (m+1-j)k+j(k+1)\\
      &=(m+1)(n+1-m)+j\\
      &=2n+j-m^2+n(m-1)+1\\
      &\geq 2n+1-m^2+(m+1)(m-1)+1=2n+1.
  \end{align*}

Without loss of generality, we assume $ \lambda _{m+1-j}< k$, which implies   $\mu_j\geq k+2$ and consequently  $f(\mu, j)=j-1$.
 Hence, we have   $\lambda_{m+1-j}+\mu_j
    >n+k-a-j+1$, with $a$ as in Lemma \ref{propbig}, iii).   Without loss of generality, we can assume
    \begin{equation}
        \lambda_{m+1-j}+\mu_j=n+k+2-j-a
    \end{equation}
    (Otherwise, we can replace $\mu$ by  $\tilde{\mu}:=(\mu_1-1,\mu_2-1,\cdots,\mu_j-1,\mu_{j+1},\cdots,\mu_m)$ by the same arguments as for \eqref{proofBeq} in the proof of Proposition \ref{thmtypeB}.)
    Thus for any $i\leq a$,
  $(i, m+1-j)$ is a big subscript pair, namely $ \lambda_i+\lambda_{m+1-j}\geq 2k+m+1-j-i
    $.
   Hence we have
    \begin{align*}
              |\lambda|+|\mu|&\geq\sum_{i=1}^{a}\lambda_i+\sum_{i=a+1}^{m+1-j}
            \lambda_i+\sum_{i=1}^{j}\mu_j\\
            &\geq\sum_{i=1}^{a}(2k+m+1-j-i-\lambda_{m+1-j})+(m+1-j-a)\lambda_{m+1-j}
            +j\mu_j+\frac{j(j-1)}{2}\\
            &=  (k+n+2-j)a-\frac{a(a+1)}{2}+(n-k+2-j-2a)(n+k+2-j-a-\mu_j)\\
             &\qquad +j\mu_j+
            \frac{j(j-1)}{2}\\
            &=J(j, a, \mu_j)
    \end{align*}
    where $J(x, y, z)=I(x, y, z)-x-y-z+n+k+3$ with $I(x, y, z)$ defined in the proof of Proposition \ref{thmtypeB}. By using the same analysis as therein, we conclude the following:
    \begin{enumerate}
        \item[(i)] If  $2a+2j-n+k-2\geq 0$, then
           $$J(j, a, \mu_j)\geq J({1\over 5}(2n-2k), {1\over 3}(2n-2k-2j), k+2)=2n+1+n(m-1)-{1\over 5}(4m^2-6m+7). $$
        Since $2\leq m<n$ for $N=2n+2$, we have
        $$n(m-1)-{1\over 5}(4m^2-6m+7) \geq (m+1)(m-1)-{1\over 5}(4m^2-6m+7)={1\over 5}(m^2+6m-12)>0.$$
     \item[(ii)] If  $2a+2j-n+k-2< 0$,
     then $J(j, a, \mu_j)\geq J(1, 0, n+k-j+1)=2n+1$.
 \end{enumerate}
In either cases, we deduce the contradiction    $2n+1>\left| \lambda \right| +\left| \mu \right|\geq 2n+1$.
\end{proof}

\begin{cor}\label{cortypeD}
    Let  $(\lambda, t_1)$ and $(\mu, t_2)$ be  in $\mathcal{P}_{2n+2}(k, n)$, and satisfy $|\lambda|+|\mu|<2n+1$.
  Then $\sP^\vee(\lambda, t_1)\preceq \sP(\mu, t_2)$ and $X_{(\lambda, t_1)}\cup X_{(\mu, t_2)}\neq 0$.
  \end{cor}

\begin{proof}
We simply denote $\sP^\vee:=\sP^\vee(\lambda, t_1)$ and $\hat{\sP}:=\sP(\mu, t_2)$. By Proposition \ref{thmtypeD} (1) and the definition of $g$ in \eqref{funcgg},
    $\lambda_{m+1-j}=k<\lambda_{m-j}$ and  $\mu_j= k<\mu_{j-1}$ cannot both hold.
Then by   property (ii)  in \eqref{eqntypepartD},
 we have  $p_i^\vee\neq n+1$ whenever $\hat p_i=n+2$.

  Set $\bar p^{\vee}_j=2n+3-p_{m+1-j}$ for all $j\in [m]$. Then $\bar p_j^\vee =p_j^\vee$ except when $n$ is even and $p_{m+1-j} \in \{n+1,n+2\}$.
 In the exception, $\bar p_j^\vee\in\{n+1, n+2\}$,  $\lambda_{m+1-j}=k<\lambda_{m-j}$  and consequently $\hat p_j\not\in \{n+1, n+2\}$ by property (ii)  in \eqref{eqntypepartD}.
  Thus  $\sP^\vee\leq \hat{\sP}$ if and only if for all $j\in [m]$, $\bar p_j^\vee\leq \hat p_j$ holds,  which is equivalent to   the   following inequality:
   \begin{align*}
&f(\lambda ,m+1-j)+f(\mu ,j)
 \leq n+k- \lambda _{m+1-j}-\mu_j
          -3+g(\lambda, t_1, m+1-j)+g(\mu, t_2, j).
          \end{align*}
 Since  $|\lambda| +|\mu|< 2n+1$, the above inequality does hold by Proposition \ref{thmtypeD}.
Hence,   $\sP^\vee(\lambda, t_1)\preceq \sP(\mu, t_2)$ holds by Proposition \ref{bruhatD}. Consequently,  $X_{(\lambda, t_1)}\cup X_{(\mu, t_2)}\neq 0$ by Proposition \ref{propnonvanishing}.
\end{proof}

\subsection{Proof of Theorem \ref{mainthm1} for classical   types }$\mbox{}$ \label{Pfclassical}

  (i) Case $\mbox{B}_n(m)$ with  $m<n$. Then $k=n-m$. By Corollary \ref{cortypeB}, $[X_\lambda]\cup [X_\mu]\neq 0$  if $|\lambda|+|\mu|<2n$. Thus
$\mbox{e.d.}(G/P_m)\geq 2n-1$. By \eqref{Chernprod}, $[X_{(1, \cdots, 1)}]\cup [X_{(n+k, 0,\cdots, 0)}]=0$. Thus
   $\mbox{e.d.}(G/P_m)<m+n+k=2n$. Hence, $\mbox{e.d.}(G/P_m)=2n-1$.

  (ii) Case $\mbox{D}_{n+1}(m)$ with  $2\leq m<n$. Then $k=n+1-m$. By Corollary \ref{cortypeD}, we have $\mbox{e.d.}(G/P_m)\geq 2n$. By   \eqref{Chernprod},
   $[X_{(1, \cdots, 1)}]\cup [X_{(n+k, 0,\cdots, 0)}]=0$. Thus
   $\mbox{e.d.}(G/P_m)<m+n+k=2n+1$. Hence, $\mbox{e.d.}(G/P_m)=2n$.

  (iii) Cases $\mbox{A}_{n}(m)$, $\mbox{B}_{n}(n)\cong\mbox{D}_{n+1}(n)\cong\mbox{D}_{n+1}(n+1)$ and $\mbox{D}_{n+1}(1)$. These together
       with cases $\mbox{C}_{n}(1)$, $\mbox{E}_{6}(1)\cong\mbox{E}_{6}(6)$ and $\mbox{E}_{7}(7)$  are a special class of Grassmannians, called \textit{minuscule Grassmannians}.
     There are nice    properties of minuscule Grassmannians. For instance, we have
         $c_1(G/P_m)=\mbox{h}(\mathcal{D})[X_{s_m}]$ (see e.g. \cite[section 2.1]{CMP}). Thus the quantum variable $\bar q_m$ in
         $QH^*(G/P_m)=H^*(G/P_m)[\bar q_m]$ has degree $$\deg \bar q_m=\langle c_1(G/P_m), [X^{s_m}]_h\rangle=\mbox{h}(\mathcal{D}),$$ the Coxeter number of the Dynkin diagram of $G$.
Therefore for any $u, v\in W^{P_m}$ with $\ell(u)+\ell(v)<\mbox{h}(\mathcal{D})$, by the nonvanishing property  (Proposition \ref{quantumnonvanishi}), we have
 $$ [X_u]\cup [X_v]=[X_u]\star [X_v]\neq 0.$$
 The equality follows immediately from the $\mathbb{Z}$-graded algebra structure of $QH^*(G/P_m)$.
 Hence $\mbox{e.d.}(\mathcal{D}(m))\geq \mbox{h}(\mathcal{D})-1$.

  \begin{enumerate}
     \item  Considering \eqref{bundleseq} for type $\mbox{A}_n$, we have $0=c_m(\mathcal{S}^\vee)\cup c_{n+1-m}(\mathcal{Q})=[X_{(1, \cdots, 1)}]\cup [X_{(n+1-m, 0,\cdots, 0)}]$. Hence $\mbox{e.d.}(\mbox{A}_n(m))<m+n+1-m=n+1= \mbox{h}(\mbox{A}_n)$.
     \item For $\mbox{B}_n(n)$, by \eqref{Chernprod} we have $[X_{(1, \cdots, 1)}]\cup [X_{(n+k, 0,\cdots, 0)}]=0$, hence $\mbox{e.d.}(\mbox{B}_n(n))<m+n+k=2n= \mbox{h}(\mbox{B}_n)$.
     \item  $\mbox{D}_{n+1}(1)$   is a quadric hypersurface $Q$ in $\mathbb{P}^{2n+1}$ of (complex) dimension  $2n=\mbox{h}(\mbox{D}_{n+1})$.
     Since $\dim H^{2n}(Q)>1$ , for any $u\in W^{P_1}$ with $\ell(u)=n$, there exists $v\in W^{P_1}$ (possibly $v=u$) with $\ell(v)=n=\dim Q-n=\ell(w_0uw_{P_1})$,  such  that $v\not \leq w_0uw_{P_1}$. It follows that $[X_u]\cup [X_v]=0$ by Proposition \ref{propnonvanishing}. (In fact $H^{2n}(Q)=\mathbb{Z}[X_{s_n\cdots s_1}]+\mathbb{Z}[X_{s_{n+1}s_{n-1}s_{n-2}\cdots s_1}]$, and the cup of such distinct Schubert classes vanishes.) Hence,  $\mbox{e.d.}(\mbox{D}_{n+1}(1))<2n =\mbox{h}(\mbox{D}_{n+1})$.
     \item $\mbox{C}_n(1)=\mathbb{P}^{2n-1}=\mbox{A}_{2n-1}(1)$. Thus $\mbox{e.d.}(\mbox{C}_n(1))=2n-1<\mbox{h}(\mbox{C}_n)$.
  \end{enumerate}
  In all cases, we have $\mbox{e.d.}(\mathcal{D}(m))<\mbox{h}(\mathcal{D})$. Hence, $\mbox{e.d.}(\mathcal{D}(m))=\mbox{h}(\mathcal{D})-1$.

 As it will be proved in the next section, we have $\mbox{e.d.}(\mbox{E}_6(6))=12=\mbox{h}(\mbox{E}_6)$ and $\mbox{e.d.}(\mbox{E}_7(7))=19=\mbox{h}(\mbox{E}_7)+1$.

(iv) Case $\mbox{C}_n(m)$. The Weyl group of types $\mbox{B}_n$ and $\mbox{C}_n$ are identical. The Schubert structure constants for complete flag variety of type
types $\mbox{B}_n$ and $\mbox{C}_n$ are the same up to a power of $2$. Write $[X_\lambda]\cup [X_\mu]=\sum_\nu e_{\lambda, \mu}^\nu [X_\nu]$ for $H^*(SG(m, 2n))$ and   $[X_\lambda]\cup [X_\mu]=\sum_\nu f_{\lambda, \mu}^\nu [X_\nu]$ for $H^*(OG(m, 2n+1))$. As a special case \cite[section 2.2]{BKT} of \cite[section 3.1]{BeSo}, we have
    $$ f_{\lambda, \mu}^\nu=2^{\ell_k(\nu)-\ell_k(\lambda)-\ell_k(\mu)}e_{\lambda, \mu}^\nu$$
    where $\ell_k(\lambda)$ denotes the number of parts $\lambda_k$ which are strictly greater than $k$. In particular, we have
        $\mbox{e.d.}(\mbox{C}_n(m))=\mbox{e.d.}(\mbox{B}_n(m))=2n-1=\mbox{h}(\mbox{C}_n)-1$.    $\hfill\square$

        \section{Proof of Theorem \ref{mainthm1} for exceptional Lie types }
 In this section, we verify Theorem \ref{mainthm1} for types E and F with the help of computer computations.
 There are three rational homogeneous varieties of type $\mbox{G}_2$ in total, which are all of very small dimension (equal to $5$ or $6$). A full table of all the (quantum) products has been obtained (see for instance \cite[Table 4]{ELM}), from which we can read off $\mbox{e.d.}(\mbox{G}_2(1))=\mbox{e.d.}(\mbox{G}_2(2))=5$ immediately.

By Proposition \ref{propnonvanishing}, it suffices to find incomparable pair $(u, w)$ (i.e. $u, v\in W^{P_m}$ with $u\neq w$) such that $\ell(u)+\ell(w_0ww_{P_m})=\ell(u)+\dim G/P_m-\ell(w)$ as small as possible. The minimum is equal to $\mbox{e.d.}(G/P_m)+1$. We interchange $G/P_m$ with the notation $\mathcal{D}(m)$ freely, where $\mathcal{D}$ is the Lie type of $G$. Effective good divisibility of  $\mbox{E}_6(1)$ and $\mbox{E}_7(7)$ (i.e. of  Cayley plane and  Freudenthal variety)  have been implicitly contained in \cite{CMP}.
Indeed, we can read off $\mbox{e.d.}(\mathcal{D}(m))$ from the relevant Hasse diagrams therein. The Hasse diagram of  $W^{P_m}$ is a graphical rendering of  the Bruhat order.
A marking point in the Hasse diagram   represents an element in $W^{P_m}$.  For a graph  drawn horizontally (resp. vertically), a marking is on the  $j$-th column (resp. $j$-th row) if and only if the corresponding element  in $W^{P_m}$ is of  length $j$. We also call such marking of length $j$.
A directed path in
 the horizontal (resp. vertical) Hasse diagram is a path in which each edge  goes from right to   left (resp. from down to up).
For any $u, w\in W^{P_m}$, $u\leq w$ if and only if there is a directed path  connecting the markings   $u, v$.

 For cases $\mbox{E}_6(1)$ and $\mbox{E}_7(7)$,  we include a part of the Hasse diagram from \cite{CMP} below, which are sufficient for our purpose. Therein  $\sigma_j$ is a Schubert class in $H^{2j}(\mathcal{D}(m))$.

  % \vspace{-0.6cm}
 % \begin{figure}[h]
 % \caption{Part of Hasse diagram of $\mbox{E}_6(1)$ and $\mbox{E}_7(7)$}\label{hasse}
        \begin{center}
    \includegraphics[scale=0.42]{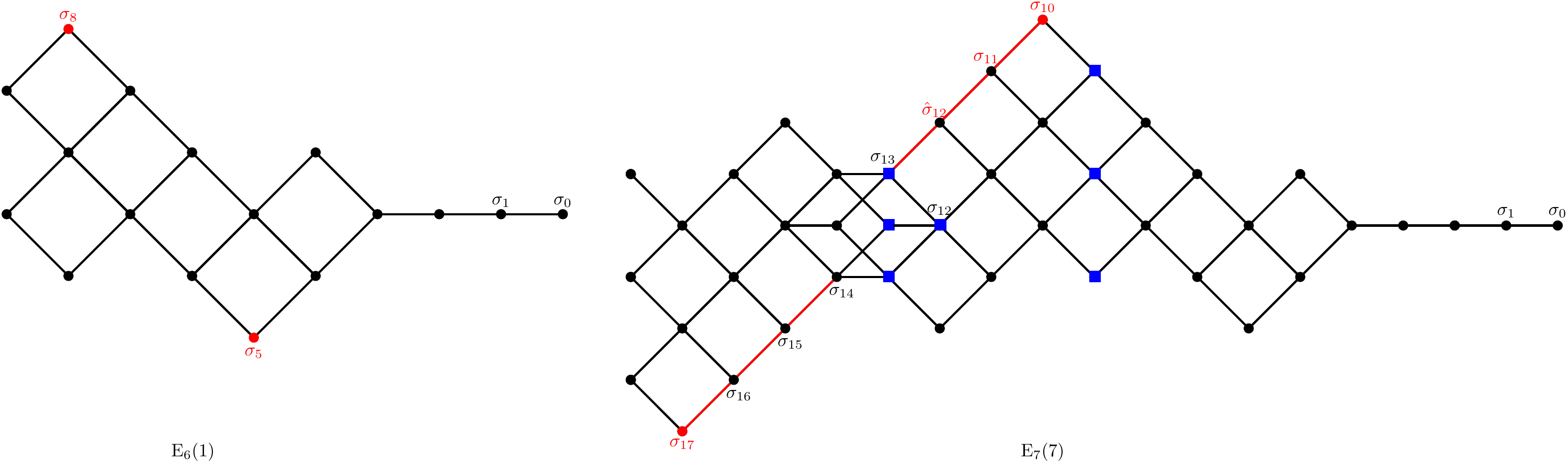}
   \end{center}
  %\end{figure}

 %\vspace{-4cm}

\noindent Indeed, in the Hasse diagram of $\mbox{E}_7(7)$, it is obvious that there does not exist any directed path connecting the markings $u=\sigma_{10}$ and $w=\sigma_{17}$ (since
  the two paths $\sigma_{10}\!\!-\!\!\!-\!\sigma_{11}\!\!-\!\!\!-\!\hat \sigma_{12}\!\!-\!\!\!-\!\sigma_{13}$ and $\sigma_{14}\!\!-\!\!\!-\!\sigma_{15}\!\!-\!\!\!-\! \sigma_{16}\!\!-\!\!\!-\!\sigma_{17}$  cannot be connected).
It follows that $u\not\leq w$ with $\ell(u)=10$ and $\ell(w)=17$. Thus for $v:=w_0ww_{P_m}\in W^{P_m}$, we have $\ell(v)=\dim \mbox{E}_7(7)-\ell(w)=27-17=10$ and
  $[X_u]\cup [X_v]=0$. Hence, $\mbox{e.d.}(\mbox{E}_7(7))<\ell(u)+\ell(v)=20$. Furthermore for any $\tilde u, \tilde v\in W^{P_m}$ with $\ell(\tilde u)+\ell(\tilde v)=19$,
  we conclude  $[X_{\tilde u}]\cup [X_{\tilde v}]\neq 0$, i.e. $\tilde u\leq w_0\tilde v w_{P_m}=:\tilde w$, or  equivalently for any
     marking $\tilde u$ of length $j$ and any marking $\tilde w$ of length $8+j$ ($=\dim \mbox{E}_7(7)-(19-\ell(\tilde u)))$ for any $1\leq j\leq 9$. This does hold
      for the fact that any marking of length $13$ or $9$ can be connected to the marking $\sigma_{12}$ of length $12$ by some directed path. For $j=0,\ldots,4$, the statement is obvious. It follows that for $5\leq j\leq 9$,
      any marking of length $j$ can be connected to another marking of length $8+j$ by a directed path (that passes the markings in squares: a marking of length 9, the marking $\sigma_{12}$ and a marking of length $13$).  For $\mbox{E}_6(1)$, we consider the incomparable pair
       $(\sigma_5, \sigma_8)$, which implies   $\mbox{e.d.}(\mbox{E}_6(1))<\ell(\sigma_5)+(\dim \mbox{E}_6(1)-\ell(\sigma_8))=5+(16-8)=13$.
        It is  obvious that  for $0\leq j\leq 6$,   any
     marking   of length $j$ and any marking   of length $4+j (=16-(12-j))$ can be connected by some directed path. In conclusion, we have
$$\mbox{e.d.}(\mbox{E}_6(1))=12,\qquad \mbox{e.d.}(\mbox{E}_7(7))=19.$$
 It is also easy to obtain  $\mbox{e.d.}(\mathcal{D}(m))$ by investigating a part of the Hasse diagram for the cases of $\mbox{F}_4(1), \mbox{F}_4(2), \mbox{E}_6(2)$ and $\mbox{E}_7(1)$. We provide them  in Figure \ref{hassediag} (starting with the divisor class $\sigma_1$) for the convenience of  the interested readers, where we have marked four
   incomparable pairs $(\sigma_5, \sigma_7), (\sigma_4, \sigma_{9}), (\sigma_4, \sigma_{10}), (\sigma_6, \sigma_{10})$. %(, which we obtained by Mathematica 10.0)
  \begin{figure}
  \caption{Part of Hasse diagram of $\mathcal{D}(m)$}\label{hassediag}
   \begin{center}
    %\includegraphics[scale=0.4]{E6graph.pdf}

  % \hspace{0.05cm} $F_4(1)$ \hspace{1.4cm}  $E_6(2)$ \hspace{1.4cm} $E_7(1)$
    \includegraphics[scale=0.40]{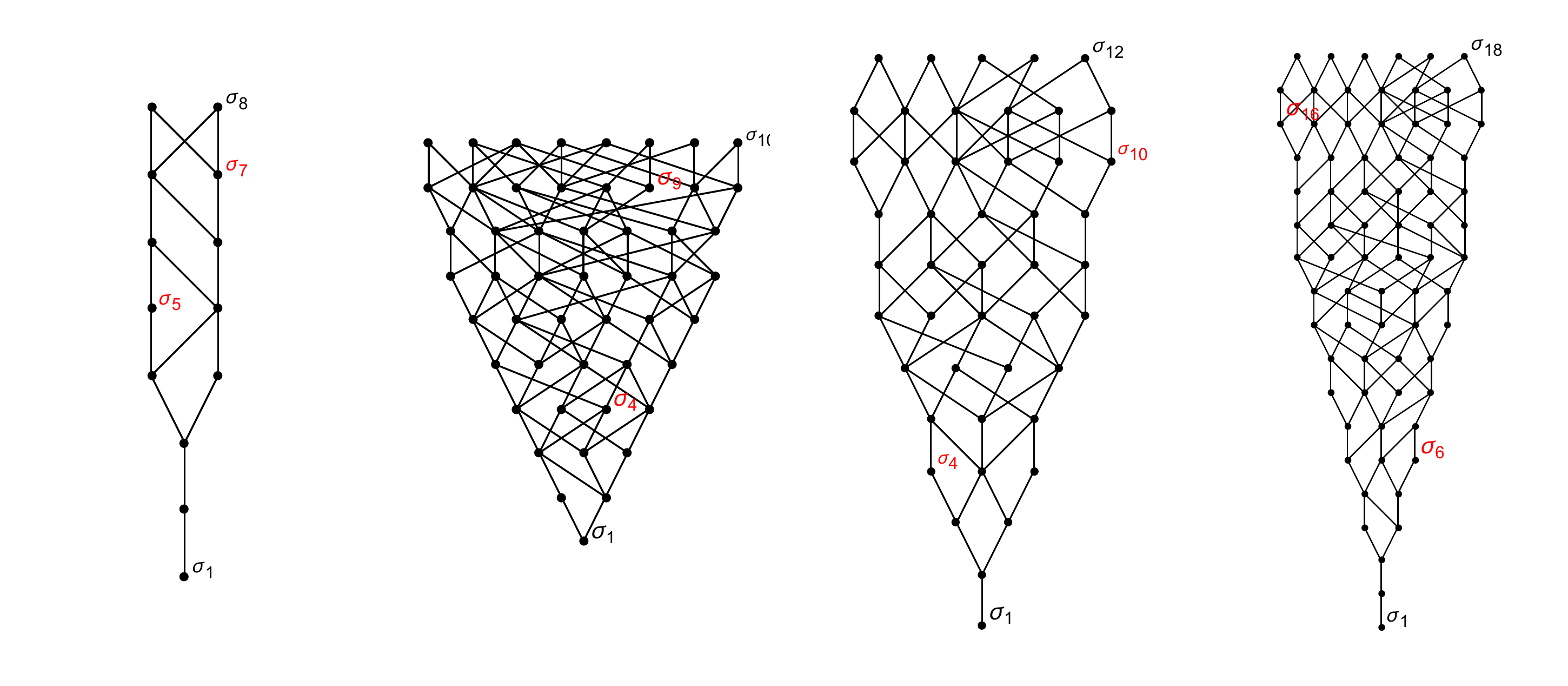}

   \hspace{-0.3cm} $F_4(1)$  \hspace{1.9cm}  $F_4(2)$  \hspace{1.9cm}  $E_6(2)$ \hspace{1.8cm} $E_7(1)$
  \end{center}
 \end{figure}

 In all cases, we can specify an incomparable pair $(u, w)$ in Table \ref{tabuw}, which results in the coincidence of $L:=\ell(u)+\ell(v)$ with
  $\mbox{e.d.}(G/P_m)+1$ (where $v=w_0ww_{P_m}$) after checking $\tilde u\leq \tilde w$ for any $\ell(\tilde u)=i$ and $\ell(\tilde w)=i+L-1-\dim \mathcal{D}(m)$, for all  $1\leq i\leq {L-1\over 2}$   (see \cite{HLL} for the codes by Mathematica 10.0). It follows that $\mbox{e.d.}(G/P_m)=\ell(u)+\ell(v)-1$.
  Comparing the quantity e.d. in Table \ref{tabed} with the quantity $L$ in  Table \ref{tabuw}, we conclude the following immediately.

  \begin{prop}
    The good effective divisibility {\upshape $\mbox{e.d.}(\mathcal{D}(m))$} for $\mathcal{D}$ of exceptional type is precisely given in Table \ref{tabed}.
 \end{prop}

As we can see from Table \ref{tabuw}, all the pairs  $(u, w)$ satisfy the property that $u^{-1}$ is again a Grassmannian permutation with respect to another $\mathcal{D}(m')$, namely $u^{-1}\in W^{P_{m'}}$.
 Thus if $u$ were less than $w$, then any substring of $w$ that gives $u$ must start with the same digit as for $u$.
 This observation, together with
     the Lifting property \cite[Proposition 2.2.7]{BjBr}of the Bruhat order, enables us to verify the incomparable pairs even without using computers.
 \begin{example}
    For   $\mbox{E}_7(3)$, we have $u^{-1}\in W^{P_7}$, $w= w'   w''$ with $ w'=s_{624534132456}$ belonging to $W_{P_7}$.
          Hence,   $$u\leq w \Longleftrightarrow u\leq  w''\Longleftrightarrow s_7 u\leq  s_7w''.$$
    Here the first equivalence follows from the two properties: (1) any reduced expression of $u$ has to start with $s_7$, by noting   $u^{-1}\in W^{P_7}$ and using the third Corollary in \cite[section 10.2]{Hump}; (2)   $w'$ does not contain   $s_7$. The second equivalence follows from  the Lifting property of the Bruhat order.
    %and ${\hat w}^{-1}\in W^{P_8}$ where   $\hat w=s_{876543245678134567245634524}$. It follows that $u\leq w \Longleftrightarrow u\leq \hat w$.
    Now $\bar u:=s_7u= s_{65432413}$ satisfies similar property to $u$, i.e. $\bar u^{-1}\in W^{P_6}$.
    Continuously using the equivalence of the above form, we conclude $$u\leq w\Longleftrightarrow s_{432413}=s_5s_6s_7u\leq  s_{532413}
    \Longleftrightarrow s_{432413}\leq s_{413}=s_4s_1s_3,$$
while the last inequality obviously fails to hold. (The last equality is our notation convention.)
 \end{example}

\begin{table}[t]
\caption{An incomparable pair $(u, w)$ for $\mathcal{D}(m)$}\label{tabuw}
\hspace{-0.5cm} \begin{tabular}{|p{0.3cm}|c| l|l| l| c|}
   \hline
  $\mathcal{D}$ & $m$ &  {}\hspace{0.7cm} $u$  &{}\hspace{1.1cm} $v$ & {}\hspace{1.9cm}$w=w_0vw_{P_m}$ &$L$\\
    \hline \hline
    &1 & $s_{12321}$ & $s_{12342321}$  &$s_{2342321}$ &13\\
   \cline{2-6}
   & 2 & $s_{1232}$ & $s_{12342312312}$  &$s_{323432312}$ &15\\
  \cline{2-6}
  \raisebox{1.5ex}[0pt]{$\mbox{F}_4$} &  3 & $s_{4323}$ & $s_{43213243243}$  &$s_{232123243}$ &15\\
  \cline{2-6}
  &  4 & $s_{43234}$ & $s_{43213234}$  &$s_{3213234}$ &13\\
    \hline
   % & $m$ & $u$ & $v$ & $v^{\#}$ &$D$\\
    \hline
    &  1 & $s_{65431}$ & $s_{13452431}$  &$s_{13452431}$ &13\\
     \cline{2-6}
     &  2 & $s_{1342}$ & $s_{65432451342}$  &$s_{3456245342}$ &15\\
     \cline{2-6}
   \raisebox{1.5ex}[0pt]{$\mbox{E}_6$} &  3 & $s_{1345243}$ & $s_{65432413}$  &$s_{43245643245432413}$ &15\\
      \cline{2-6}
    & 4 & $s_{654324}$ & $s_{1345624534}$  & $s_{5341324564132451324}$ &16\\
     \hline
   % & $k$ & $u$ & $v$ & $v^{\#}$ &\\
     \hline
   & 1 & $s_{765431}$ & $s_{76543245613452431}$ & $s_{6543245613452431}$&23\\
     \cline{2-6}
   & 2 & $s_{765432451342}$ &$s_{765432451342}$&$s_{564534132456734132456432451342}$&24 \\
     \cline{2-6}
   & 3 & $s_{765432413}$ & $s_{7654324561345243}$ & $s_{6245341324567541324561324532413}$&25  \\
     \cline{2-6}
   $\mbox{E}_7$ & 4 & $s_{76543245134}$   & $s_{765432456134524}$ & $s_{45624534132456745341324563413245341324}$&26    \\
     \cline{2-6}
   & 5 & $s_{76543245}$ & $s_{765432456713456245}$ & $s_{45624534132456745341324563413245}$&26   \\
    \cline{2-6}
   & 6 & $s_{765432456}$ & $s_{765432456713456}$ & $s_{345624534132456724534132456}$&24  \\
     \cline{2-6}
   & 7 & $s_{7654324567}$ & $s_{7654324567}$ & $s_{13456245341324567}$&20   \\

      \hline
\end{tabular}

   \begin{tabular}{|p{0.3cm}|c|l| l| c|}
   \hline
   $\mathcal{D}$ & $m$ &{}\hspace{1.5cm} $u$ &{}\hspace{1.9cm} $v$ &$L$ \\
     \hline
    & 1 & $s_{876543245613452431}$ & $s_{87654324567813456724563452431}$ &47  \\
      \cline{2-5}
    & 2 & $s_{8765432456713456245342}$ & $s_{87654324567813456724563451342}$ &51  \\
        \cline{2-5}
    & 3 &  $s_{87654324567134562453413}$ & $s_{8765432456781345672456345243}$ &51 \\
      \cline{2-5}
    & 4 & $s_{8765432456713456245341324}$ & $s_{876543245678134567245634524}$ &52\\
       \cline{2-5}
   $\mbox{E}_8$ & 5 & $s_{8765432456781345672456345}$ & $s_{87654324567134562453413245}$ &51\\
       \cline{2-5}
    & 6 & $s_{8765432456781345672456}$ & $s_{876543245671345624534132456}$ &49 \\
       \cline{2-5}
    &7 &  $s_{876543245678134567}$ & $s_{8765432456713456245341324567}$ &46 \\
     \cline{2-5}
    &8 &  $s_{876543245678}$ & $s_{87654324567134562453413245678}$ &41 \\
     \hline
   \hline
   &    &\multicolumn{3}{|c|}{   $w=w_0vw_{P_m}$}  \\
     \hline
  &  $1$ & \multicolumn{3}{|l|}{$s_{2456734562453413245678245341324567543245613452431}$ }     \\
      \cline{2-5}
   &  2 &   \multicolumn{3}{|l|}{$s_{543245671345624534132456783456245341324567245341324565432451342}$ }  \\
        \cline{2-5}
   & 3 &   \multicolumn{3}{|l|}{$s_{2453413245678543245671345624534132456783456245341324567654324561345243}$ } \\
     \cline{2-5}
   &  4 &  \multicolumn{3}{|l|}{ $s_{6543245671345624534132456713456245341324565432451342876543245678134567245634524}$ }  \\
      \cline{2-5}
  $\mbox{E}_8$ &  5 &  \multicolumn{3}{|l|}{$s_{765432456713456245341324567564534132456341324543241387654324567134562453413245}$ }\\
      \cline{2-5}
    & 6 & \multicolumn{3}{|l|}{ $s_{7654324567134562453413245678654324567134562453413245671345624534132456}$} \\
       \cline{2-5}
   & 7 &  \multicolumn{3}{|l|}{$s_{7654324567134562453413245678765432456713456245341324567}$ }  \\
      \cline{2-5}
   & 8 &  \multicolumn{3}{|l|}{ $s_{7654324567134562453413245678}$ } \\
     \hline
\end{tabular}

\end{table}

\section{Morphisms to rational homogeneous varieties}
In this section, we provide  applications of good effective divisibility on   the non-existence of non-constant morphisms from complex projective manifolds to rational homogeneous varieties. We will prove the main application Theorem \ref{mainthm3} right  after
      Proposition \ref{morphismprop}, which cares about  morphisms to Grassmannians.

Recall  $Q_{2n}= SO(2n+2, \mathbb{C})/P_{1}$ is a smooth quadric hypersurface in $\mathbb{P}^{2n+1}$.
The proof of the following proposition is from \cite[Corollary 5.3]{MOS22}.
\begin{prop}\label{morphismQ} Let $M$ be a connected complex projective manifold.
If {\upshape $\mbox{e.d.}(M)\geq 2n$}, then any morphism $\varphi: M\to Q_{2n}$  is a constant map.
\end{prop}
\begin{proof} For $G/P_1=Q_{2n}$, we let $u=s_{n}\cdots s_2s_1$ and $w=s_{n+1}s_{n-1}\cdots s_2s_1$. Then we have $u, w\in W^{P_1}$,
   $X_u\cap X^w=\emptyset$, $w_0uw_{P_1}=w$  and $[X_u]=[X^w]$. Hence,  %$=[X_u]\cup [X^w]=0$.
     $$[\varphi^{-1}(X_u)]\cup [\varphi^{-1}(X^w)]=\varphi^*([X_u])\cup \varphi^*([X^w])=\varphi^*([X_u]\cup [X^w])=0.$$
    Since $\mbox{e.d.}(M)\geq 2n=2\ell(u)$, we have $[\varphi^{-1}(X_u)]=0\in H^{2n}(M)$, so that $\varphi$ is not surjective. { Composing with
     a linear projection from a linear subspace $\Lambda$ of $\mathbb{P}^{2n+1}$ disjoint from $\varphi(M)$ of maximal dimension, we obtain a surjective morphism
      $\hat \varphi: M\to \mathbb{P}^{\dim \varphi(M)}$. } Since
      $\mbox{e.d.}(M)\geq 2n> \mbox{e.d.}(\mathbb{P}^{\dim \varphi(M)})=\dim \varphi(M)$,
      it follows that  $\dim\varphi(M)=0$. Hence, $\varphi$ is a constant.
\end{proof}

The next proposition strengthens  \cite[Proposition 4.7]{MOS20}, generalizing \cite[Corollary 3.2]{Tango} and \cite[Proposition 2.4]{NaOc}, while the proofs   are essentially the same as theirs. Recall our notation convention $G/P_m=\mathcal{D}(m)$ for $G$ of type $\mathcal{D}$.
\begin{prop}\label{morphismprop} Let $M$ be a connected complex projective manifold, and $\mathcal{D}(m)$  be a Grassmannian of classical type.
 If {\upshape $\mbox{e.d.}(M)>\mbox{e.d.}(\mathcal{D}(m))$}, then any morphism $\varphi: M\to \mathcal{D}(m)$  is   constant.
 %and $\varphi: M\to G/P$ be a morphism with  $G$   of classical type.    If $G/P_m\not\cong Q_{2n}$
%and $\mbox{e.d.}(M)>e.d.(G/P)$, then  $\varphi$ is a constant map.
\end{prop}
\begin{proof}
 If $\mathcal{D}(m)=D_{n+1}(1)$, then we are done by Proposition \ref{morphismQ}.

 If $\mathcal{D}$ is of type $D_{n+1}$ and  $m\in \{n, n+1\}$, then we consider $B_n(n)$ instead, due to the isomorphisms $D_{n+1}(n)\cong D_{n+1}(n+1)\cong B_{n}(n)$.

For the remaining cases, we always consider  the tautological  vector bundles $\mathcal{S}^\vee$ and $\mathcal{Q}$ over $\mathcal{D}(m)$. They have effective nonzero Chern classes
  (see e.g. \cite{Tamv, BKT}):
        \begin{align*}
           c(\mathcal{S}^\vee)&=\sum_{i=0}^mc_i(\mathcal{S}^\vee)=\sum_{i=1}^m [X_{s_{m-i+1}\cdots s_{m-1}s_m}], \\
           c(\mathcal{Q})&=\sum_{i=0}^{N-m}c_i(\mathcal{Q})=\sum_{i=1}^{N-m} a_i [X_{v(i)}].
        \end{align*}
    Here $v(i)$ denotes the element in $W^{P_m}$ that corresponds to the ($k$-strict) partition $(i, 0, \cdots, 0)$; $N=\mbox{e.d.}(\mathcal{D}(m))+1$ by \eqref{Chernprod} and Theorem \ref{mainthm1}, and $a_i\in \{1, 2\}$ is  up to the type of $G$ and whether $i>n-m$. Then we can write $c(\varphi^*(\mathcal{S}^\vee))=1+\eta_1+\cdots+\eta_m$ and $c(\varphi^*(\mathcal{Q}))=1+\xi_1+\cdots+\xi_{N-m}$
       with $\eta_i, \xi_i\in H^{2i}(M)$ being effective classes.
       Since $1=c(\mathcal{S})\cup c(\mathcal{Q})$, we have $c_m(\mathcal{S}^\vee)\cup c_{N-m}(\mathcal{Q})=0$, implying
           $\eta_m \cup \xi_{N-m}=0$. Since $N=\mbox{e.d.}(\mathcal{D}(m))+1\leq \mbox{e.d.}(M)$, by definition we have
            $\eta_m=0$ or $\xi_{N-m}=0$. By induction we conclude that either $\eta_1=0$ or $\xi_1=0$ must hold. Notice
             $\eta_1=c_1(\varphi^*\det \mathcal{S}^\vee)$ and
        $\xi_1=c_1(\varphi^*\det \mathcal{Q})$.
                  We consider $\mathcal{L}= \det \mathcal{S}^\vee$ if $\eta_1=0$, or $\mathcal{L}=\det \mathcal{Q}$ otherwise. In either cases, $\mathcal{L}$ is an ample line bundle over $\mathcal{D}(m)$ with  $c_1(\varphi^* \mathcal{L})=0$.
                  Assume that $\varphi$ is not a constant morphism. Then $\varphi(x)\neq \varphi(y)$ for some points $x, y\in M$. Take an irreducible curve $C$ in $M$ passing through $x$ and $y$.  Then $\varphi(C)$ is of dimension one, and $\varphi_*([C])$ is a nonzero effective curve class in $H_2(\mathcal{D}(m), \mathbb{Z})$.   By projection formula, we have
              $$0=\varphi_*(c_1(\varphi^*\mathcal{L})\cap [C])=c_1(\mathcal{L})\cap \varphi_*[C]>0.$$
         This implies a contradiction.
              Hence,   $\varphi$ is a constant morphism.
        \end{proof}
\bigskip

 \begin{proof}[Proof of Theorem \ref{mainthm3}]
   By Theorem \ref{mainthm2},  we have  $\mbox{e.d.}(G/P)=\mbox{e.d.}(G/P_m)$ for some   $\alpha_m\in \Delta\setminus \Delta_P$. Recall that $\pi_m: G/P\to G/P_{m}=\mathcal{D}(m)$ denotes the natural projection. Then $\pi_m\circ \varphi: M\to \mathcal{D}(m)$ is a morphism with $\mbox{e.d.}(M)>\mbox{e.d.}(\mathcal{D}(m))$.
   Therefore $\pi_m\circ \varphi$ is a constant morphism by Proposition \ref{morphismprop}. That is, $\varphi(M)$ is inside a fiber of $\pi_m$. 
   
   The Dynkin diagram of  $\Delta_{P_m}=\Delta\setminus\{\alpha_m\}$ consists of $r$ connected components of type $\mathcal{D}^{(1)}, \cdots, \mathcal{D}^{(r)}$ respectively.
 Observe that
     the fiber $P_m/P$ is isomorphic to  a product $X_1\times \cdots \times X_r$ with $X_i=G^{(i)}/P^{(i)}$ being  a  rational homogeneous variety (possibly a point)   of type $\mathcal{D}^{(i)}$.
  As long as $X_r$ is not a point, we consider the composition $\hat \pi$ of the natural projections $X_1\times\cdots \times X_r\to X_r\to  G^{(r)}/\hat P=: {\mathcal{D}^{(r)}}(\hat m)$, where the $\hat P$ is a maximal parabolic subgroup of $G^{(r)}$ containing $P^{(r)}$.
  Clearly, either of the following cases must hold:   (1) $\mathcal{D}^{(r)}$ has the same Lie type as $\mathcal{D}$, and is of  rank  $|\mathcal{D}^{(r)}|<|\Delta|$; (2)  $\mathcal{D}^{(r)}$ is of type $A$,   $\mbox{h}(\mathcal{D}^{(r)})<\mbox{h}(\mathcal{D})-1$. Consequently, we always have $\mbox{e.d.}(G/P)>\mbox{e.d.}(\mathcal{D}^{(r)}(\hat m))$. Hence,  $\hat \pi\circ \varphi(M)$ is again a constant morphism  by Proposition \ref{morphismprop}. That is, 
   $\varphi(M)$ is  inside the fiber $\hat \pi$.
    
  By considering   iterated fibrations of $G/P$ and using induction, we conclude that $\varphi$ is a constant morphism.
       \end{proof}

\begin{remark}\label{rmkmorphism}
The notion {\upshape $\mbox{e.d.}(M)$} can be naturally extended in the setting of Chow rings $A^*(M)$ as in \cite{MOS22}. The above propositions can also be directly
  generalized to morphisms from  (possibly singular) projective varieties $M$. The proof of Proposition \ref{morphismprop} will also work, by adding one sentence  `` By Kleiman's transversality theorem, there exist  $g, g'\in G$, such that $\varphi^{-1}(g X_{s_{m-i+1}\cdots s_{m-1}s_m})$'s and
   $\varphi^{-1}(g'X_{v(i)})$'s are all generically reduced and of the same codimension as the corresponding Schubert varieties; this ensures that
     $\varphi^*([X_{s_{m-i+1}\cdots s_{m-1}s_m}])=[\varphi^{-1}(g X_{s_{m-i+1}\cdots s_{m-1}s_m})]$
and $\varphi^*([X_{v(i)}])=[\varphi^{-1}(g X_{v(i)})]$ are all effective classes in   $A^i(M)$".
\end{remark}
As a consequence of Theorem \ref{mainthm3},  we obtain  Corollary \ref{corsubdiag} and the follow proposition, by directly    comparing  $\mbox{e.d.}(G/P_m)$'s using Theorem \ref{mainthm1}.
\begin{cor}
   Let $G, \tilde G$ be of the same classical type with $\mbox{rank}(G)<\mbox{rank}(\tilde G)$. For any parabolic subgroups $P\subset G$ and $\tilde P\subset \tilde G$, there does not exist any non-constant morphism $\tilde G/\tilde P\to G/P$.
\end{cor}

\end{document}